\documentclass[10pt,reqno]{amsart}

\usepackage{amsthm,amsmath,amssymb}
\usepackage{graphicx}
\usepackage{float}
\usepackage[colorlinks=true,
linkcolor=blue,
urlcolor=blue,
citecolor=blue]{hyperref}
\usepackage{mathtools}
\usepackage{relsize}
\usepackage{ytableau}
\usepackage{tikz}
\usepackage{enumerate}
\usepackage[toc,page]{appendix}
\usepackage{lscape}
\usepackage{multirow}
\usepackage{adjustbox,expl3,etoolbox} 
\usepackage[margin = 3.2cm]{geometry}
\usepackage{diagbox}

\newtheorem{thm}{Theorem}[section]
\newtheorem{lem}[thm]{Lemma}
\theoremstyle{definition}

\newtheorem{prop}[thm]{Proposition}
\newtheorem{conj}[thm]{Conjecture}
\newtheorem{cor}[thm]{Corollary}

\newtheorem{rmk}[thm]{Remark}

\newtheorem{qst}[thm]{Question}

\DeclareMathOperator\sym{Sym}
\DeclareMathOperator\alt{Alt}
\DeclareMathOperator{\fix}{fix}

\DeclareMathOperator\Der{Der}

\DeclareMathOperator\gl{GL}
\DeclareMathOperator\soc{Soc}

\DeclareMathOperator{\pg}{PG}

\DeclareMathOperator{\psl}{PSL}
\DeclareMathOperator{\ps}{PSp}
\DeclareMathOperator{\po}{P\Omega^\varepsilon}
\DeclareMathOperator{\pgl}{PGL}
\DeclareMathOperator{\sln}{SL}
\newcommand{\mathieu}[1]{\operatorname{M}_{#1}}

\letcs\replicate{prg_replicate:nn} \newcommand*\longsum[1][1]{% 
	\mathop{\textnormal{% 
			\clipbox{0pt 0pt {.5\width} 0pt}{$\displaystyle\sum$}%
			\replicate{#1}{\clipbox{{.5\width} 0pt {.4\width} 0pt}{$\displaystyle\sum$}}% 
			\clipbox{{.6\width} 0pt 0pt 0pt}{$\displaystyle\sum$}}}% 
}

\begin{document}	
	
	\title[]{On the Intersection density of primitive groups of degree a product of two odd primes}
	
	\author[A.S. Razafimahatratra]{Andriaherimanana Sarobidy Razafimahatratra}
	%\thanks{\textsuperscript{*} University of Primorska, UP FAMNIT, Glagolja\v{s}ka 8, 6000 Koper, Slovenia}
	\address{University of Primorska, UP FAMNIT, Glagolja\v{s}ka 8, 6000 Koper, Slovenia}\email{sarobidy@phystech.edu}
	
	\begin{abstract}
		A subset $\mathcal{F}$ of a finite transitive group $G\leq \sym(\Omega)$ is intersecting if for any $g,h\in \mathcal{F}$ there exists $\omega \in \Omega$ such that $\omega^g = \omega^h$. The \emph{intersection density} $\rho(G)$ of $G$ is the maximum of $\left\{ \frac{|\mathcal{F}|}{|G_\omega|} \mid \mathcal{F}\subset G \mbox{ is intersecting} \right\}$, where $G_\omega$ is the stabilizer of $\omega$ in $G$. In this paper, it is proved that if $G$ is an imprimitive group of degree $pq$, where $p$ and $q$ are distinct odd primes, with at least two systems of imprimitivity then $\rho(G) = 1$. Moreover, if $G$ is primitive of degree $pq$, where $p$ and $q$ are distinct odd primes, then it is proved that $\rho(G) = 1$, whenever the socle of $G$ admits an imprimitive subgroup.
	\end{abstract}

	
	
	\date{\today}
	
	\maketitle
	
	%\tableofcontents
	
	\section{Introduction}
	%%%%%%%%%%%%%%%%%%%%%
	
	Let $\Omega$ be a finite set and $G\leq \sym(\Omega)$ be a finite transitive group. A subset $\mathcal{F} \subset G$ is \emph{intersecting} if any two permutations of $\mathcal{F}$ agree on at least one element of $\Omega$. That is, for all $\sigma,\pi \in \mathcal{F}$, there exists $\omega \in \Omega$ such that $\omega^\sigma = \omega^\pi$. If $\mathcal{F} \subset G$ is an intersecting set, then its \emph{intersection density} is the rational number 
	$\rho(\mathcal{F}) := \frac{|\mathcal{F}|}{|G_\omega|},$
	where $\omega \in \Omega$. The \emph{intersection density } of the group $G$ is the rational number $$\rho(G) := \max\left\{ \rho(\mathcal{F}) : \mathcal{F} \subset G \mbox{ is intersecting} \right\}.$$ The intersection density of groups was first introduced in \cite{li2020erd}.
	Note that $\rho(G)\geq 1$ since a point-stabilizer of $G$ is an intersecting subset of $G$. We say that the transitive group $G\leq \sym(\Omega)$ has the {\it Erd\H{o}s-Ko-Rado property} or EKR property if $\rho(G) = 1$. Moreover, we say that $G$ has the {\it strict-EKR property} if it has the EKR property and the only intersecting sets with intersection density equal to $1$ are cosets of point-stabilizers. 
	
	The study of transitive groups having the EKR property started with the 1977 paper of Deza and Frankl \cite{Frankl1977maximum}. It was proved in this paper that $\sym(n)$ has the EKR property. In 2004, Cameron and Ku \cite{cameron2003intersecting}, independently Larose and Malvenuto \cite{larose2004stable}, proved that $\sym(n)$ has the strict-EKR property. Since then, many works on the EKR property of transitive groups have appeared in the literature \cite{ahmadi2014new,ahmadi2015erdHos,meagher2015erdos,bardestani2015erdHos,ellis2012setwise,ellis2011intersecting,godsil2009new,long2018characterization,li2020erd,meagher2019erdHos,meagher2016erdHos,meagher2014erdos,plaza2015stability,spiga2019erdHos}. Examples of groups having the EKR property are finite doubly transitive groups \cite{meagher2016erdHos} and transitive groups admitting sharply transitive sets. Examples of groups that do not have the EKR property are given in \cite{2021arXiv210803943H,meagher20202,AMC2554}. The following conjecture on the intersection density of transitive groups of certain degrees was posed in \cite{meagher180triangles}.
	\begin{conj} 
		Let $G\leq \sym(\Omega)$ be a transitive group.
		\begin{enumerate}[(a)]
			\item If $|\Omega|$ is a prime power, then $\rho(G) = 1$.\label{conj1}
			\item If $|\Omega| = 2p$, where $p$ is an odd prime, then $\rho(G)\leq 2$. Moreover, this upper bound is tight for any odd prime $p$.\label{conj2}
			\item If $|\Omega| = pq$, where $p$ and $q$ are distinct odd primes, then $\rho(G) = 1$.\label{conj3}
		\end{enumerate}\label{conj-main}
	\end{conj}
	
	Conjecture~\ref{conj-main} \eqref{conj1} was recently proved by Li et al. \cite{li2020erd}, and independently, Hujdurovi\'c et al. \cite{hujdurovic2021intersection}.
	
	\begin{thm}[\cite{hujdurovic2021intersection,li2020erd}]
		If $G$ is a transitive group of prime power degree, then $\rho(G) = 1$.
	\end{thm}
	
	In this paper, we are interested in the intersection density of transitive groups of degree $pq$, where $p$ and $q$ are distinct primes with $p>q$. In \cite{AMC2554}, it was proved that when $q=2$, the intersection density of transitive groups of degree $2p$ is in the set $[1,2]\cap \mathbb{Q}$. In \cite{meagher180triangles,AMC2554}, it was proved that there exists a transitive group of degree $2\ell$, for any odd $\ell$, with intersection density equal to $2$. These results settle Conjecture~\ref{conj-main}\eqref{conj2}.

	A question raised in \cite[Question~6.1]{AMC2554} is whether a transitive group $G$ of degree $2p$ always has integral intersection density, that is, $\rho(G) \in \{1,2\}$. This question was recently answered by Hujdurovi\'c et al. in \cite{hujdurovic2021intersection}.
	\begin{thm}[\cite{hujdurovic2021intersection}]
		If $G$ is a transitive group of degree $2p$, where $p$ is an odd prime, then $\rho(G) \in \{1,2\}$.
	\end{thm}
	For the rest of this paper, we suppose that $G$ is a transitive group of degree $pq$, where $p$ and $q$ are odd primes with $p>q$. Hujdurovi\'c et al. \cite{2021arXiv210709327H} recently proved that Conjecture~\ref{conj-main}\eqref{conj3} fails for certain imprimitive groups of degree $pq$. An example of groups for which the conjecture fails is \verb|TransitiveGroup(33,18)| in the library of transitive groups in \verb*|Sagemath| \cite{sagemath}. Further, it was proved in \cite{2021arXiv210709327H} that if an imprimitive group of degree $pq$ has a block of size $p$, then it has the EKR property. We will see in fact that the imprimitive groups violating Conjecture~\ref{conj-main}\eqref{conj3} have exactly one system of imprimitivity, with blocks of size $q$. 
	
	In this paper, we prove that Conjecture~\ref{conj-main}\eqref{conj3} holds for imprimitive groups of degree $pq$ with at least two systems of imprimitivity.
	Our first main result is stated as follows. 
	
	\begin{thm}
		If  $G\leq \sym(\Omega)$ is an imprimitive group of degree $pq$, where $p$ and $q$ are distinct odd primes, with at least two different systems of imprimitivity then $\rho(G) = 1$.\label{thm:imprimitive}
	\end{thm}
	
	Next, we consider the primitive cases (see \cite[Section~2]{2021arXiv210709327H} for a comprehensive description of these groups). If $G$ is doubly transitive of degree $pq$, then by the main result of \cite{meagher2016erdHos}, $\rho(G) = 1$. Therefore, we may assume that $G$ is \emph{simply primitive} (i.e., a primitive group which is not doubly transitive). 
	Recall that if $G$ is a group, then the \emph{socle} of $G$, denoted $\soc(G)$, is the subgroup generated by all the minimal normal subgroups of $G$. If $G \leq \sym(\Omega)$ is primitive, then its socle $\soc(G)$ must be transitive since it is normal in $G$ and $G$ is faithful. Using \cite[Lemma~6.5]{meagher180triangles}, we deduce that $\rho(G)\leq \rho(\soc(G))$. Therefore, in order to prove Conjecture~\ref{conj-main}\eqref{conj3} for the primitive cases, it is enough to prove that the possible socles of primitive groups of degree $pq$ have the EKR property. 
	The socles of primitive groups of degree $pq$ have been classified by Maru\v{s}i\v{c} and Scapellato \cite{maruvsivc1994classifying}. This classification is given in Table~\ref{table:classification}. Among the possible socles of $G$, there are seven families.

	 Our second main result concerns the intersection density of primitive groups of degree $pq$; in particular, those in lines 9-11. Instead of proving directly that these groups  have the EKR property, we will prove a more general result. We will then deduce as a corollary of this result that the groups in line 9-11 have the EKR property.   In the next two theorems, we state the result in question.

	\begin{thm}
		For $n\geq 16$, the intersection density of $\alt(n)$ in its actions on the $2$-subsets of $\{1,2,\ldots,n\}$ is $1$.\label{thm:alt-intro}
	\end{thm}
	
	\begin{thm}
		Let $k$ be a prime power. The intersection density of $\psl(2,k^2)\leq \ps(4,k)$ in its action on the $1$-dimensional subspaces of $\mathbb{F}_k^4$
		is equal to $1$.\label{thm:sl2-intro}
	\end{thm}
	
	We deduce the following theorem as a consequence of Theorem~\ref{thm:alt-intro} and Theorem~\ref{thm:sl2-intro}.
	
	\begin{table}[t]
		\tiny
		\begin{tabular}{|c|c|c|c|c|}
			\hline
			Line & $\soc(G)$ & $(p,q)$ & action & Information\\
			\hline\hline
			1& $\alt(7)$ & $(7,5)$ & triples &  \\
			\hline 
			2& $\psl(4,2)$ & $(7,5)$ & $2$-spaces & \\
			\hline
			3& $\psl(5,2)$ & $(31,5)$ & $2$-spaces & \\
			\hline
			4& $\psl(2,23)$ & $(23,11)$& cosets of $\sym(4)$ & \\
			\hline
			5& $\psl(2,11)$ & $(11,5)$ & cosets of $\alt(4)$ & \\
			\hline
			6& $\mathieu{11}$ & $(11,5)$	 & & \\
			\hline
			7& $\mathieu{22}$ & $(11,7)$ & & \\
			\hline
			8& $\mathieu{23}$ & $(23,11)$ & & \\
			\hline
			9& $\alt(p)$ & $\left(p,\frac{p-1}{2}\right)$ & pairs & $p\geq 5$\\
			\hline
			10& $\alt(p+1)$ & $\left(p,\frac{p+1}{2}\right)$ & pairs & $p\geq 5$\\
			\hline
			11& $\ps(4,k)$ & $(k^2+1,k+1)$ & $1$-spaces & $p,q$ are Fermat primes\\
			\hline
			12& $\po (2d,2)$ & $\left(2^d - \varepsilon,2^{d-1} + \varepsilon\right)$ & $\begin{aligned}
				\mbox{singular}\\
				\mbox{ 1-spaces}
			\end{aligned}$ & $\begin{aligned}
				&\varepsilon = 1 \mbox{ and $d$ is a Fermat prime} \\ 
				&\varepsilon = -1 \mbox{ and $d-1$ is a Mersenne prime}
			\end{aligned}$\\
			\hline
			13& $\psl(2,p)$ & $\left( p,\frac{p+1}{2} \right)$ & cosets of $D_{p-1}$ & $p\geq 13$ and $p\equiv 1 (mod\ 4)$\\
			\hline
			14& $\psl(2,p)$ & $\left(p,\frac{p-1}{2}\right)$ & cosets of $D_{p+1}$ & $p\geq 13$ and $p\equiv 3 (mod\ 4)$\\
			\hline
			15& $\psl(2,q^2)$ & $\left( \frac{q^2+1}{2},q \right)$ & cosets of $\pgl(2,q)$ & \\
			\hline
			16& $\psl(2,p)$ & $(19,3),(29,7),(59,29)$ & cosets of $\alt(5)$& \\
			\hline
			17 & $\psl(2,13)$ & $(13,7)$& cosets of $\alt(4)$ & missing in \cite{maruvsivc1994classifying}, see \cite{du2018hamilton} \\
			\hline
			18& $\psl(2,61)$ & $ (61,31)$ & cosets of $\alt(5)$& \\
			\hline
		\end{tabular}
		\caption{Socles of simply primitive groups $G$ of degree $pq$.}\label{table:classification}
	\end{table}

	\begin{thm}
		If $G$ is primitive of degree $pq$ and $\soc(G)$ is one of the groups in lines 1-11,14,16,17 of Table~\ref{table:classification}, then $\rho(G) = 1$. In particular, if $\soc(G)$ contains an imprimitive group, then $\rho(G) = 1$.\label{thm:main}
	\end{thm}

	This paper is organized as follows. In Section~\ref{sect:background}, we recall some preliminary results on derangement graphs and the conjugacy class scheme of a group. The proofs of Theorem~\ref{thm:imprimitive} and Theorem~\ref{thm:main} are given in Section~\ref{sect:imprimitive} and Section~\ref{sect:primitive}. 
	The rest of the paper is dedicated to the proofs of Theorem~\ref{thm:alt-intro} and Theorem~\ref{thm:sl2-intro}.
	
	\section{Background results}\label{sect:background}
	In this section, we give a brief review of the EKR theory of transitive permutation groups. We let $G\leq \sym(\Omega)$ be a finite transitive group throughout this section. 
	\subsection{Derangement graphs}
	Given a group $K$ and a subset $S \subset K\setminus \{1\}$, the \emph{Cayley digraph} $\operatorname{Cay}(K,S)$ of $G$ with \emph{connection set} $S$ is the digraph whose vertex set is $G$ and two elements $x,y\in K$ are connected with the arc $(x,y)$ if $yx^{-1} \in S$. If the connection set $S$ has the property that $x\in S$ implies $x^{-1}\in S$, then the $\operatorname{Cay}(K,S)$ is an undirected graph. The undirected Cayley graph $\operatorname{Cay}(K,S)$ is regular with valency $|S|$ and is also vertex transitive since the right-regular representation of $K$ is a regular subgroup of the full automorphism group of $\operatorname{Cay}(K,S)$. Moreover, the number of components of such a graph is equal to the index of the subgroup $\langle S\rangle$ in $K$. For more details, see \cite{godsil2001algebraic}. 
	
	A derangement of $G$ is a permutation without a fixed point. A well-known result due to Camille Jordan \cite{jordan1872recherches} asserts that a finite transitive group of degree at least $2$ always has a derangement. Let $\Der(G)$ be the set of all derangements of $G$. The \emph{derangement graph} $\Gamma_G$ of $G$ is the graph whose vertex set is $G$ and two group elements $g,h\in G$ are adjacent if $hg^{-1} \in \Der(G)$. It is not hard to see that $\Gamma_G$ is the Cayley graph of $G$ with connection set equal to $\Der(G)$. Consequently, $\Gamma_G$ is regular of valency $|\Der(G)|$ and is vertex transitive. Moreover, $\Gamma_G$ is a \emph{normal Cayley graph} since $\Der(G)$ is the union of conjugacy classes of derangements, which means it is invariant under conjugation. It is worth mentioning that the same terminology is also used for Cayley graphs on a group $G$ with the property that the right-regular representation of $G$ is a normal subgroup of the full automorphism group of the Cayley graph (see \cite{xu1998automorphism} for details).
	
	For any $g,h\in G$ and $\omega \in \Omega$, we have $\omega^g = \omega^h \Leftrightarrow \omega = \omega^{hg^{-1}}$; in other words, $hg^{-1} \not\in \Der(G)$, meaning that $g$ and $h$ are non-adjacent in $\Gamma_G$. Consequently, $\mathcal{F} \subset G$ is intersecting if and only if it is a \emph{coclique} or an \emph{independent set} of $\Gamma_G$.
	Therefore, the problem of studying the maximum intersecting sets of $G\leq \sym(\Omega)$ reduces to the study of the maximum cocliques of $\Gamma_G$.
	
	\subsection{Maximum cocliques}
	Let $X = (V,E)$ be an undirected graph.  A \emph{clique} in $X$ is a subgraph of $X$ in which every pair of vertices are adjacent; that is, a complete subgraph of $X$. A \emph{coclique} of $X$ is a subgraph of $X$ in which every pair of vertices are non-adjacent; that is, an empty induced subgraph of $X$. The maximum size of a clique and coclique of $X$ are denoted by $\omega(X)$ and $\alpha(X)$, respectively.
	
	If $X = (V,E)$ is a vertex-transitive graph (i.e., its full automorphism group acts transitively on $V(X)$), then one can obtain a bound on the size of the maximum cocliques of $X$ using the maximum cliques. This bound is given in the next lemma.
	\begin{lem}[Clique-coclique Bound \cite{godsil2016erdos}]
		If $X= (V,E)$ is vertex transitive, then $\omega(X) \alpha(X) \leq |V(X)|$. Equality holds if and only if every clique of $X$ intersects every coclique of $X$ at a unique vertex.
	\end{lem}
	\begin{cor}
		Let $G\leq \sym(\Omega)$. If $G$ admits a regular subgroup, then $\rho(G)=1$.
	\end{cor}
	\begin{proof}
		Let $H$ be a regular subgroup of $G$. For any $\omega,\omega^\prime \in \Omega$, there is exactly one element of $H$ that maps $\omega$ to $\omega^\prime$. Consequently, no two permutations of $H$ agree on an element of $\Omega$. In other words, $H$ is a clique in the derangement graph $\Gamma_G$ of size $|H|=|\Omega|$. By the  Clique-coclique bound, we have $\alpha(\Gamma_G)\leq \frac{|G|}{|\Omega|}$ and so $\rho(G) = 1$.
	\end{proof}
	
	The next result gives an upper bound on the size of the maximum cocliques  using an algebraic method. If $S\subset V(X)$, then the characteristic vector $v_S$ of $S$ is the $\{0,1\}$-vector of $\mathbb{Z}^{|V(X)|}$ indexed by $V(X)$ whose $s$-entry is equal to $1$ if $s \in S$, and $0$ otherwise. Let $\mathbf{1}$ be the vector whose entries are all equal to $1$ (its number of entries should be clear from the context).
	\begin{lem}[Hoffman bound - Ratio Bound \cite{haemers2021hoffman}]
		Let $X = (V,E)$ be a regular graph with degree equal to $d$ and minimum eigenvalue $\tau$. Then, 
		\begin{align*}
			\alpha(X) \leq \frac{\tau}{\tau-d}|V(X)|.
		\end{align*}
		Moreover, if equality holds and $\mathcal{C}$ is a maximum coclique of $X$, then the translated characteristic vector $v_\mathcal{C} - \frac{|\mathcal{C}|}{|V(X)|}\mathbf{1}$ is an eigenvector with eigenvalue $\tau$.
	\end{lem}
	
	A \emph{weighted adjacency matrix} corresponding to a graph $X = (V,E)$ is a $|V(X)|\times |V(X)|$ real symmetric matrix $A$ whose row sum is constant and the entry $A_{u,v} = 0$ if $u\not\sim_X v$. The following result is a refinement of the Ratio Bound.
	\begin{lem}[Weighted Ratio Bound \cite{godsil2016erdos,haemers2021hoffman}]
		Let $X = (V,E)$ be a regular graph and let $A$ be a weighted adjacency matrix of $X$. Suppose that $d$ and $\tau$ are respectively the row sum and the minimum eigenvalue of $A$. Then,
		\begin{align*}
			\alpha(X) \leq \frac{\tau}{\tau -d} |V(X)|.
		\end{align*}\label{lem:weighted-ratio-bound}
	\end{lem}

	\subsection{The conjugacy class scheme}
	Throughout this subsection we let $G$ be an abstract group and $\mathcal{C}$ be the set of all conjugacy classes of $G$. We say that a  conjugacy class $C$ of $G$ is inverse-closed  if $x^{-1} \in C$, for any $x\in C$.
	
	For any conjugacy class $C\in \mathcal{C}$, define the $\{0,1\}$-matrix $A_C$ whose rows and columns are indexed by group elements and $\left(A_C\right)_{g,h}  = 1$ if and only if $hg^{-1} \in C$. The \emph{conjugacy class scheme} of $G$ is the association scheme obtained from the set of matrices $\mathcal{A}(G) = \left\{A_C \mid C\in \mathcal{C}\right\}$. See \cite{godsil2016erdos} for details on these combinatorial objects. If $K$ is a finite transitive group, then it is not hard to see that any linear combination of matrices corresponding to conjugacy classes of derangements from $\mathcal{A}(K)$ is a weighted adjacency matrix of the derangement graph $\Gamma_{K}$.
	
	If the conjugacy classes of $G$ are all inverse-closed, then every matrix in $\mathcal{A}(G)$ is symmetric, that is, the conjugacy class association scheme of $G$ is symmetric. Moreover, the matrices in $\mathcal{A}(G)$ commute with each other, so the matrices in $\mathcal{A}(G)$ are simultaneously diagonalizable. Therefore, an eigenvalue of a linear combination of matrices of $\mathcal{A}(G)$ is a linear combination of eigenvalues of matrices of $\mathcal{A}(G)$.
	
	If $G$ has a conjugacy class which is not inverse-closed, then at least one of the matrices in $ \mathcal{A}(G)$ is not symmetric. It is still possible to prove that the eigenvalues of a linear combination of matrices of $\mathcal{A}(G)$ are linear combinations of eigenvalues of the matrices in $\mathcal{A}(G)$ (see \cite[Section~3.4]{godsil2016erdos}). Due to the existence of the idempotents of $\mathcal{A}(G)$ (see \cite[Theorem~3.4.4]{godsil2016erdos}), the matrices in $\mathcal{A}(G)$ are still diagonalizable. Since the matrices of $\mathcal{A}(G)$ are pairwise commuting, they are simultaneously diagonalizable. Hence, we obtain similar properties as when the matrices of $\mathcal{A}(G)$ are symmetric.
	
	Therefore, the matrices of $\mathcal{A}(G)$ admit a common basis $\mathcal{B}$ of eigenvectors. Let $v\in \mathcal{B}$. For any $C \in \mathcal{C}$, let $\lambda_C$ be the eigenvalue of $A_C$ corresponding to the eigenvector $v$. The discussion in the two previous paragraphs leads to the following straightforward result.
	
	\begin{lem}
		Let $v\in \mathcal{B}$. If $A = \sum_{C\in \mathcal{C}} k_C A_C$ is a linear combination of $\mathcal{A}(G)$, then $v$ is an eigenvector of $A$ with eigenvalue
		$\sum_{C\in \mathcal{C}} k_C \lambda_C.$\label{lem:general-eigenvalues-association-scheme}
	\end{lem}
	
	For any $C\in \mathcal{C}$, the matrix $A_C$ $\in \mathcal{A}(G)$ is the adjacency matrix of the Cayley digraph of $G$ with connection set equal to $C$. The spectrum of such a Cayley graph can be determined by a result of Babai in \cite{babai1979spectra}. From this, the spectrum of the derangement graph of a transitive group $G$ can be found. Recall that a Cayley graph is \emph{normal} if its connection set is invariant under conjugation. The following result gives the eigenvalues of normal Cayley graphs.
	
	\begin{lem}[Babai \cite{babai1979spectra}]
		Let $X = \operatorname{Cay}(G,D)$ be a Cayley graph such that $D $  is invariant under conjugation. (i.e., a normal Cayley graph). Let $ (\mathfrak{X}_1,V_1),\ (\mathfrak{X}_2, V_2),\ldots,\ (\mathfrak{X}_k,V_k )$ be a complete list of distinct irreducible representations of $G$ and let $\chi_i$ be the character afforded by $\mathfrak{X}_i$, for any $i\in \{1,2,\ldots,k\}$. Then,
		\begin{align*}
			\mathbb{C}G \cong \bigoplus_{i=1}^k  U_i,
		\end{align*}
		where $U_i$ is the sum of all submodules  isomorphic to $V_i$ in the regular $\mathbb{C}G$-module.
		Moreover, $U_i$ is a subspace of the eigenspace of $X$ with eigenvalue 
		\begin{align*}
			\xi_{\mathfrak{X}_i} = \frac{1}{\dim \mathfrak{X}_i} \sum_{g\in D} \chi_i(g),
		\end{align*}
		for any $i \in \{1,2,\ldots,k\}$. The dimension of the eigenspace corresponding to $\xi_{\mathfrak{X}_i}$ is equal to 
		\begin{align*}
			\sum_{\{j \mid \xi_{\mathfrak{X}_j} = \xi_{\mathfrak{X}_i}\}} \chi_j(1)^2.
		\end{align*}
		\label{lem:eig-norm}
	\end{lem}

	Lemma~\ref{lem:eig-norm} gives us an important tool to compute the eigenvalues of a normal Cayley graph. However, most of our results require the use of weighted adjacency matrices. Next, we present a remarkable result which generalizes Lemma~\ref{lem:eig-norm} to weighted adjacency matrices. This result is due to Diaconis and Shahshahani \cite[Corollary~3]{diaconis1981generating}.
	
	\begin{lem}
		Let $ D_1,D_2,\ldots,D_k $ be all conjugacy classes of derangements of a transitive group $G$. For $i\in \{1,2,\ldots,k\}$, let $g_i$ be an arbitrary element of $D_i$. Consider the weighted adjacency matrix 
		$
			A= \sum_{i=1}^k \omega_i A_{D_i}.
		$ %.
		The eigenvalues of $A$ are of the form
		\begin{align*}
			\frac{1}{\chi(1)}\sum_{i=1}^k \omega_i\chi(g_i)|D_i|,
		\end{align*}
		where $\chi$ runs through the irreducible characters of $G$.\label{lem:general-eigenvalue-formula-scheme}
	\end{lem}
	
		\section{Imprimitive cases}\label{sect:imprimitive}
	
	In this section, we prove Theorem~\ref{thm:imprimitive}. Suppose that $G \leq \sym(\Omega)$ is imprimitive of degree $pq$, where $p>q$ are odd primes and has at least two systems of imprimitivity. We will use a classification result with respect to the number of systems of imprimitivity, which is due to Lucchini \cite{lucchini1991imprimitive}.
	
	\begin{thm}[Lucchini \cite{lucchini1991imprimitive}]
		Let $G\leq \sym (\Omega)$ be an imprimitive group of degree $pq$, where $p>q$ are odd primes. Let $m$ be the number of systems of imprimitivity of $G$. If $m\geq 2$, then either $m=2$ or $m = p+1$. In particular, we have the following:
		\begin{enumerate}[(i)]
			\item there is at most one system of imprimitivity with blocks of size $p$;
			\item if $G$ has a block of imprimitivity of size $p$ and one of size $q$, then $G\leq \sym(p) \times \sym(q)$,\label{case2}
			\item if $G$ has at least two systems of imprimitivity of size $q$, then $q \mid (p-1)$ and one of the following holds:\label{case3}
			\begin{enumerate}
				\item $G$ is a non-abelian group of order $pq$ having $p$ systems of imprimitivity with blocks of size $q$ and one with blocks of size $p$;\label{case3a}
				\item $p = 7$, $q = 3$, and $G = \psl(2,7)$ such that $G$ has exactly two systems of imprimitivity with blocks of size $3$,\label{case3b}
				\item $p = 11$, $q = 5$, and $G = \psl(2,11)$ such that $G$ has exactly two systems of imprimitivity with blocks of size $5$.\label{case3c}
			\end{enumerate}
		\end{enumerate}\label{thm:imprimitive-classification}
	\end{thm}
	
	We use Theorem~\ref{thm:imprimitive-classification} to prove that $G$ has the EKR property. First, we note that Hujdurovi\'c et al. \cite[Proposition~2.1]{2021arXiv210709327H} recently proved that if $G$ has blocks of size $p$ then $\rho(G)=1$. This deals with \eqref{case2} (and also all groups with a unique system of imprimitivity with blocks of size $p$).

	Suppose that $G$ has at least two systems of imprimitivity with blocks of size $q$. In the case \eqref{case3a}, $G$ is a non-abelian group of order $pq$. In other words, $G$ is a regular group, which implies that $\rho(G)=1$.
	
	If \eqref{case3b} holds, then $G$ is the transitive group \verb|TransitiveGroup(21,14)| of the library of transitive groups of \verb|Sagemath| \cite{sagemath}. This group has a regular subgroup isomorphic to $C_7\rtimes C_3$, so it has the EKR property. Similarly, if \eqref{case3c} holds, then $G = \psl(2,11)$ has a regular subgroup isomorphic to $C_{11}\rtimes C_5$. Therefore, $G = \psl(2,11)$ has the EKR property.

	We conclude that any imprimitive group of degree $pq$ with at least two systems of imprimitivity has the EKR property. Therefore, the remaining imprimitive groups of degree $pq$ to consider are those with exactly one system of imprimitivity, with blocks of size $q$.

	\section{Primitive cases}\label{sect:primitive}
	Let $G \leq \sym(\Omega)$ be a primitive group of degree $pq$. If $G$ is doubly transitive, then $G$ has the EKR property (see \cite{meagher2016erdHos}). Hence, we may assume that $G$ is simply primitive (i.e., primitive but not doubly transitive).
	Let $S := \soc(G)$ be the socle of $G$. Since $S$ is the subgroup generated by the minimal normal subgroups of $G$, it is easy to see that $S \trianglelefteq G$. As $G$ is primitive, $S$ is transitive. The following lemma and its corollary are crucial to the main results of this paper.
	\begin{lem}[No-Homomorphism Lemma \cite{albertson1985homomorphisms}]
		Let $X = (V(X),E(X))$ be a graph and let $Y= (V(Y),E(Y))$ be a vertex-transitive graph. If there exists a graph homomorphism $X \to Y$, then 
		\begin{align*}
			\frac{\alpha(Y)}{|V(Y)|} \leq \frac{\alpha(X)}{|V(X)|}.
		\end{align*}\label{lem:no-hom}
	\end{lem}
	
	In \cite{meagher180triangles}, the intersection density was defined for any arbitrary finite permutation group. As this will be needed in the statement of the next result, we recall this definition. If $K\leq \sym(\Omega)$ is a finite permutation group, then its intersection density is the number
	\begin{align*}
		\rho(K) = \max \left\{ \frac{|\mathcal{F}|}{\displaystyle\max_{\omega \in \Omega} |K_\omega|} : \mathcal{F} \mbox{ is an intersecting set of $K$} \right\}. 
	\end{align*}
	As a corollary of Lemma~\ref{lem:no-hom}, we obtain an upper bound on the intersection density of a transitive group from the intersection density of some of its subgroups.
	\begin{cor}
		Let $H\leq \sym(\Omega)$ be a transitive group and $K$ be a subgroup of $H$ with $k$ orbits of the same size. Then, we have 
		$
			\rho(H) \leq \rho(K)k.
		$\label{cor:no-hom}
	\end{cor}
	\begin{proof}
		As $K\leq H$, the derangement graph $\Gamma_{K} = \operatorname{Cay}(K,\Der(K))$ is an induced subgraph of $\Gamma_H$. Therefore, there exists a natural homomorphism $\Gamma_{K} \to \Gamma_H$. As $\Gamma_H$ is vertex transitive, by Lemma~\ref{lem:no-hom}, we have
		\[
			\rho(H)= \frac{\alpha(\Gamma_H)}{\frac{|H|}{|\Omega|}}\leq \frac{\alpha(\Gamma_{K}) |H||\Omega|}{|H||K|} = \frac{\alpha(\Gamma_{K})k|\omega^K|}{|K|} = \rho(K)k.\ \ \qedhere
		\]
	\end{proof}
	A consequence of Corollary~\ref{cor:no-hom} is that if $S\leq G$ is transitive and has the EKR property, then so does $G$. Therefore, to prove the primitive case of Conjecture~\ref{conj-main}\eqref{conj3}, it is enough to prove that the socles of all primitive groups of degree $pq$ have the EKR property. 	
	
	We use the classification of the socles of simply primitive groups of degree $pq$ from \cite{maruvsivc1994classifying} to determine  their EKR property. The classification is given in Table~\ref{table:classification}.
	
	\subsection*{The straightforward cases}
	We use \verb|Sagemath| \cite{sagemath} and \verb*|Gurobi| \cite{gurobi} to verify that the groups in lines 1-8 and 16-17 have the EKR property.
	\begin{enumerate}[$\bullet$]
		\item The group $\alt(7)$ of degree $35$ in line 1 has point-stabilizers of size $72$ which is equal to the upper bound given by the Ratio Bound. Therefore, $\alt(7)$ of degree $35$ has the EKR property. Similarly, the values given by the Ratio Bound  on the derangement graphs of the groups in line 2 and line 16 with $p = 19$, are equal to the sizes of their respective point-stabilizers. 
		
		\item The group $\psl(5,2)$ in line 3 has a regular subgroup isomorphic to $C_{31}\rtimes C_5$. Therefore, it has the EKR property. Similarly, the groups in lines 4-6, 8 and 16 when $p \in\{29, 59\}$ have regular subgroups. Therefore, they have the EKR property.
		
		\item The Mathieu group $\mathieu{22}$ in line~7 does not admit a regular subgroup. Moreover, the Ratio Bound yields an upper bound larger than the order of a point-stabilizer. However, we can rephrase the problem into a linear programming problem (see \cite{gl2} for details) and find a suitable weighted adjacency matrix via \verb*|Gurobi| for which the weighted Ratio Bound yields the order of a point-stabilizer. Hence, $\mathieu{22}$ has the EKR property.
		
		\item The group $\psl(2,13)$ of degree acting on cosets of $\alt(4)$ in line~17 has intersection density equal to $1$ via direct computation of the independence number of its derangement graph on \verb*|Sagemath|. It is worth noting that none of the methods  described in the first three bullets work for this group. In particular, there exists no weighted adjacency matrices from the conjugacy class scheme of $\psl(2,13)$ that give the right number in the weighted Ratio Bound.
		
		\item All methods used in the first three bullets also fail for the group $\psl(2,61)$ in line~18. In addition, we could not compute directly the independence number of the derangement graph of this group on \verb*|Sagemath| due to its size. The optimal upper bound on the intersection density given by \verb*|Gurobi| is $2.08$ for weighted adjacency matrices from the conjugacy class scheme of $\psl(2,61)$.
		
	\end{enumerate}	
	\subsection*{Infinite families}
	\begin{enumerate}[$\bullet$]
		\item Next, we consider the groups in line 14. When $p \equiv 3 \mod 4$, the group $K = C_p\rtimes C_{\frac{p-1}{2}}$ is a regular subgroup of $\psl(2,p)$ with the action described in line 14. Therefore, the group in line 14 has the EKR property (see \cite{li2020erd} for details). 
		
		\item Now we turn our attention to  the groups in lines 9-10. 	For $\left(p,\frac{p-1}{2}\right) \in \{ (7,3),(11,5) \}$, it is easy to see that $\alt(p)$ acting on the $2$-subsets of $\{1,2,\ldots,p\}$ contains a regular subgroup, thus, $\rho(\alt(p)) = 1$. For $\left(p,\frac{p+1}{2}\right) \in \{ (5,3),(13,7) \}$, we verified on \verb*|Sagemath| that the action of $\alt(p+1)$ on the $2$-subsets of $\{1,2,\ldots,p+1\}$ has the EKR property; thus $\rho(\alt(p+1)) = 1.$ This settles the EKR property for the groups in lines~9 and 10 for $p\leq 17$.
		
		For the case $n\geq 16$, we prove a more general statement on the action of $\alt(n)$ on the $2$-subsets of $\{1,2,\ldots,n\}$. We prove that for any $n\geq 16$, $\alt(n)$ acting on the $2$-subsets of $\Omega = \{1,2,\ldots,n\}$ has the EKR property (see Theorem~\ref{thm:alt-intro}). This can be further reformulated in terms of the natural action of $\sym(n)$ on $\Omega$. We say that $\mathcal{F} \subset \sym(n)$ is $2$\emph{-setwise intersecting} if for any two permutations $g,h \in \mathcal{F}$, there exists a $2$-subset $S$ of $\Omega$ such that is $S^g = S^h$. It is not hard to see that proving that the group $\alt(n)$ acting on the $2$-subsets of $\Omega$ has the EKR property is equivalent to proving that if $\mathcal{F} \subset \alt(n)$ is $2$-setwise intersecting, then $|\mathcal{F}| \leq \frac{|\alt(n)|}{\binom{n}{2}} = (n-2)!$. We prove the following. 
		
		\begin{thm}
			For $n\geq 19$, if $\mathcal{F}\subset \alt(n)$ is $2$-setwise intersecting, then $|\mathcal{F}| \leq (n-2)!$. In particular, $\rho(\alt(n)) = 1$.\label{thm:altn}
		\end{thm}
		The proof of this theorem is given in Section~\ref{sect:EKR-for-sym}. 
		
		\item For the group in line~11, it was shown in \cite{maruvsivc1994classifying} that $\ps(4,k)$ admits an imprimitive subgroup. We will prove that this subgroup has intersection density $1$. 
		
		Consider the matrix $B = \begin{bmatrix}
			0 & I_2\\
			-I_2 & 0
		\end{bmatrix}$. For any prime power $r$, the \emph{symplectic group} $\operatorname{Sp}(4,r)$ is the subgroup of $\gl(4,r)$ that consists of matrices $A$ such that 
	\begin{align}
		A^TBA = B. \label{eq:matrix-prop}
	\end{align}
	For any prime power $r$, recall that the group $\operatorname{Sp}(2,r^2)$ embeds in $\operatorname{Sp}(4,r)$. As $\operatorname{Sp}(2,r^2) = \sln(2,r^2)$, we conclude that $\sln(2,r^2)$ is a subgroup of $\operatorname{Sp}(4,r)$. Hence, $\psl(2,r^2)\leq \ps(4,r)$; this subgroup acts transitively on the $1$-dimensional subspaces of $\mathbb{F}_r^4$. If $ k$ is a power of $2$, then $\psl(2,k^2) = \sln(2,k^2)$ is a subgroup of $\operatorname{Sp}(4,k)$$=\ps(4,k)$. Note that if $k$ is a power of $2$ and $k+1$ and $k^2+1$ are Fermat primes, then $\ps(4,k)$ is the group in line~11. 
	
	Now, we will need some information on how the group $\psl(2,k^2)$ embeds in $\ps(4,k)$, for any prime power $k$. 
	Let $f(t) = t^2 + at+b$ be an irreducible polynomial over $\mathbb{F}_k$. Let $M$ be the companion matrix of $f(t)$ and $\alpha$ be a root of $f(t)$. By definition, $f(t)$ is the characteristic polynomial of $M$ and by the Cayley-Hamilton theorem, we have $f(M) = 0$. Now, we identify the field $\mathbb{F}_k$ with the field $K_1$ of $\mathbb{F}_k$-multiples of the $2\times 2$ identity matrix over $\mathbb{F}_k$. With this identification, it is easy to see that $\mathbb{F}_{k^2} = \mathbb{F}_k(\alpha)$ can also be identified with the field $K_2 :=K_1(M) = \left\{ A+MB \mid A,B\in K_1 \right\}$. In \cite{maruvsivc1994classifying}, it has been proved that the group $\sln(2,k^2)$ can be identified with the set of all matrices $\begin{bmatrix}
		A & B\\
		C &D
	\end{bmatrix}$, such that $A,B,C,D \in K_2$ and $AD - CB = I_2$, where $I_n$ is the $n\times n$ identity matrix.
	It is easy to see that the map $a_1 + a_2\alpha \mapsto a_1 I_2 + a_2 M$ from  $\mathbb{F}_k(\alpha)$ to $ K_1(M)$ gives an embedding of $\sln(2,k^2)$ in $\sln(4,k)$. As the matrices in this subgroup satisfy \eqref{eq:matrix-prop} (see \cite{maruvsivc1994classifying} for details), the subgroup $\sln(2,k^2)$ embeds in $\operatorname{Sp}(4,k)$. Hence, we also obtain an embedding of $\psl(2,k^2)$ into $\ps(4,k)$.

		In Section~\ref{sect:sln-even} and Section~\ref{sect:sln-odd}, we prove the following theorem which is more general than what is needed for the groups in line~11. 
		\begin{thm}
			Let $k$ be a prime power. If $\mathcal{F} \subset \psl(2,k^2)$ in its action on the $1$-spaces of $\mathbb{F}_k^4$, then $|\mathcal{F}| \leq k^2(k-1)$. In particular, $\rho(\psl(2,k^2)) = 1$.\label{thm:sl2}
		\end{thm}
		Using Corollary~\ref{cor:no-hom} and Theorem~\ref{thm:sl2}, we derive the following corollary.
		\begin{cor}
			The group $\ps(4,k)$ in line~11 has intersection density equal to $1$.
		\end{cor}
	\end{enumerate}

	\noindent The remainder of this paper is dedicated to the proof of Theorem~\ref{thm:altn} and Theorem~\ref{thm:sl2}.

	\section{Representation Theory of the symmetric group}\label{sect:representation-theory}
	In this section, we give a brief review on the representation theory of the symmetric group. For more detailed results on the representation theory of the symmetric group, we refer the reader to  \cite{sagan2001symmetric}.

	Recall that the irreducible $\mathbb{C}\sym(n)$-modules are the Specht modules. These irreducible representations of $\sym(n)$ are indexed by the partitions of $n$. For more details on this, see \cite[Section~2.3]{sagan2001symmetric}. 
	Given a partition $\lambda \vdash n$, the irreducible character of $\sym(n)$ corresponding to the Specht module of $\lambda$ is denoted by $\chi^\lambda$. The \emph{dimension} $f^\lambda$ of the Specht module corresponding to $\lambda$ is the value of $\chi^\lambda$ evaluated on the identity permutation of $\sym(n)$. This number is also the degree of the character $\chi^\lambda$. The degree of $\chi^\lambda$ can be computed using the Hook Length Formula \cite[Section~3.10]{sagan2001symmetric}. 
	
	Next, we present a recursive method to compute the (irreducible) character values of $\sym(n)$.
	For any $\sigma\in \sym(n)$ with cycle type $\rho = [a_1,a_2,\ldots,a_k]$ and $\lambda \vdash n$, we let $\chi_\rho^\lambda := \chi^\lambda(\sigma)$.
	The Murnaghan-Nakayama Rule is a combinatorial tool with which the character values of $\sym(n)$ can be computed. To state the Murnaghan-Nakayama rule, we need to introduce additional definitions.
	
	Let $\lambda = \left[ \lambda_1,\lambda_2,\ldots,\lambda_k  \right] \vdash n$ and $\mu = \left[\mu_1,\mu_2,\ldots,\mu_l\right]\vdash m$, where $m< n$ and $l\leq k$. We say that $\lambda$ \emph{contains} $\mu$, and write $\mu \subset \lambda$ if $\mu_i \leq \lambda_i$, for all $i\in \{1,2,\ldots,l\}$. When $\mu \subset \lambda$, the corresponding \emph{skew diagram} $\lambda/\mu$ is the set of cells of $\lambda$ that are not in $\mu$. That is, 
	\begin{align*}
		\lambda/ \mu = \left\{ c\in \lambda \mid c\not\in \mu \right\}.
	\end{align*} 
	A \emph{rim hook} $\zeta$ of a Young diagram $\lambda \vdash n$ is a skew diagram of $\lambda$ whose cells are on a path in which each step is upward or rightward. The \emph{leg length} $\ell\ell(\zeta)$ of the rim hook $\zeta$ of $\lambda$ is the number of rows it spans minus $1$ and the length $|\zeta|$ of $\zeta$ is the number of cells it has. 
	
	Given a partition $\lambda$ and a rim hook $\zeta$ of $\lambda$, we define $\lambda\setminus \zeta$ to be the set of cells of $\lambda$ that are not in $\zeta$. Since $\zeta$ is a skew diagram, $\lambda \setminus \zeta$ is a Young diagram with $n - |\zeta|$ cells.
	
	Given a partition $\rho = [a,\rho_2,\ldots,\rho_k]$ of $n$, we define the operation $\rho \setminus a$ to be the partition of $n-a$ obtained by deletion of the first part (which is $a$) of $\rho$.
	
	\begin{lem}[Murnaghan-Nakayama Rule \cite{sagan2001symmetric}]
		Let $\lambda \vdash n$ and $\rho = [\rho_1,\rho_2,\ldots,\rho_k]\vdash$$n$. We have
		\begin{align*}
			\chi^\lambda_{\rho} &= \longsum_{\zeta \in {RH_{\rho_1}}(\lambda)} (-1)^{\ell\ell(\zeta)} \chi_{\rho \backslash \rho_1}^{\lambda\setminus \zeta},
		\end{align*}
		where ${ RH}_{\rho_1}(\lambda)$ is the set of all rim hooks of length $\rho_1$ of $\lambda$.\label{lem:MurnNakRule}
	\end{lem}
	
	Using the Murnaghan-Nakayama Rule, we compute the maximum values that the irreducible characters can take on the conjugacy classes with cycle types $[n],$   $[n-1,1]$, $[n-3,3]$, $[n-4,3,1]$, and $[n-5,4,1]$.
	First, we recall the following result, which is proved in \cite{behajaina20203}.
	\begin{lem}
		Let $k,n\in \mathbb{N}$ such that $3k+1\leq n$ and $\lambda\vdash n$. Then, the Young diagram $\lambda$ has at most one rim hook of length $n-k$.\label{lem:uniqueness-rim-hook}
	\end{lem}
	
	\begin{lem}
		Let $n\geq 16$. For any $\lambda\vdash n$ and $\sigma\in \sym(n)$ with cycle type that is one of $[n],$   $[n-1,1]$, $[n-3,3]$, $[n-4,3,1]$, and $[n-5,4,1]$, we have $\chi^\lambda(\sigma) \in \{ -1,0,1 \}$.\label{lem:character-values}
	\end{lem}
	\begin{proof}
		Suppose that the cycle type of $\sigma$ is one of $[n],$   $[n-1,1]$, $[n-3,3]$, $[n-4,3,1]$, and $[n-5,4,1]$. Therefore, by the Murnaghan-Nakayama Rule and Lemma~\ref{lem:uniqueness-rim-hook}, we have
		\begin{align*}
			\left|\chi^\lambda_{[n]}\right| &\leq \max_{\nu\vdash 0}|\chi^\nu_\varnothing| = 1\\
			\left|\chi^\lambda_{[n-1,1]}\right| &\leq \max_{\nu\vdash 1}|\chi^\nu_{[1]}| = 1\\
			\left|\chi^\lambda_{[n-3,3]}\right| &\leq \max_{\nu\vdash 3}|\chi^\nu_{[3]}| = 1\\
			\left|\chi^\lambda_{[n-4,3,1]}\right| &\leq \max_{\nu\vdash 4}|\chi^\nu_{[3,1]}| = 1\\
			\left|\chi^\lambda_{[n-5,4,1]}\right| &\leq \max_{\nu\vdash 5}|\chi^\nu_{[4,1]}| = 1.
		\end{align*}
		The proof follows since the symmetric group has integral character values.
	\end{proof}
	
	\section{EKR property of $\alt(n)$ acting on $2$-subsets}\label{sect:EKR-for-sym}
	
	In this section, we refine the result of Meagher and the author in \cite{meagher20202} on the $2$-setwise intersecting permutations of $\sym(n)$ to prove Theorem~\ref{thm:alt-intro} (Theorem~\ref{thm:altn}).  
	
	We will find a weighted adjacency matrix with which we can prove that if $\mathcal{F}\subset \sym(n)$ is $2$-setwise intersecting, then $|\mathcal{F}|\leq 2(n-2)!$. Such weighted adjacency matrices are known to exist \cite{meagher20202}, however, the weighted adjacency matrix that we give in this paper enables us to prove that $2$-setwise intersecting sets in the alternating group $\alt(n)$ have size at most $(n-2)!$.
	
	If $\rho \vdash n$, then we let $A_\rho$ be the matrix of the conjugacy class scheme $\mathcal{A}(\sym(n))$ that corresponds to the conjugacy class of $\sym(n)$ with cycle type $\rho$. Consider the weighted adjacency matrix
	\begin{align}
		A &= x_{[n]}A_{[n]} + x_{[n-1,1]}A_{[n-1,1]} + x_{[n-3,3]}A_{[n-3,3]} + x_{[n-4,3,1]} A_{[n-4,3,1]} + x_{[n-5,4,1]} A_{[n-5,4,1]}.\label{w-mat}
	\end{align}
	
	The irreducible constituents of the permutation character of $\sym(n)$ acting on the $2$-subsets of $\Omega$ are $\chi^{[n]},\ \chi^{[n-1,1]}$ and $\chi^{[n-2,2]}$ (see \cite{meagher20202}).  The main idea for our proof is to find real numbers $x_{[n]},x_{[n-1,1]},x_{[n-3,3]},x_{[n-4,3,1]},$ and $x_{[n-5,4,1]}$ for which the matrix $A$ satisfies the following properties.
	\begin{enumerate}[(i)]
		\item The largest eigenvalue of $A$ is $\binom{n}{2}-1$, with multiplicity equal to $2$.\label{max}
		\item The smallest eigenvalue of $A$ is $-1$, which is afforded by the irreducible characters $\chi^{[n-1,1]}$ and $\chi^{[n-2,2]}$.\label{min}
		\item All other eigenvalues of $A$ are in the interval $\left[-1,\binom{n}{2}-1\right]$.\label{others}
	\end{enumerate}
	
	\begin{rmk}\hfil
		\begin{enumerate}[(a)]
			\item For any $\rho \vdash n$, the matrix $A_\rho$ is a symmetric matrix since the conjugacy classes of $\sym(n)$ are inverse-closed. We denote by $C_\rho$ the conjugacy class of $\sym(n)$ with cycle type $\rho$.
			\item The only non-trivial normal subgroup of $\sym(n)$ is $\alt(n)$, for $n\geq 5$. Consequently, if $n\geq 5$, then any normal Cayley graph $\Gamma = \operatorname{Cay}(\sym(n),S)$ has at most $2$ components. Moreover, if $\Gamma$ has $2$ components, then its connection set $S$ consists of even permutations and $\Gamma$ is isomorphic to the disjoint union of two copies of $\operatorname{Cay}(\alt(n),S)$. 
			
			\item We note that if the eigenvalue $\binom{n}{2}-1$ of $A$ has multiplicity $2$, then there exists a spanning subgraph $\Gamma$ of $\Gamma_{\sym(n)}$ with $2$ components for which $A$ is also a weighted adjacency matrix. We can choose $\Gamma$ to be the Cayley graph of $\sym(n)$ with connection set $S$ consisting of the union of all conjugacy classes $C_\rho$ with $x_\rho \neq 0$, for $\rho \in \{ [n],[n-1,1],[n-3,3],[n-4,3,1],[n-5,4,1] \}$.
			
			\item Therefore, if \eqref{max} is satisfied, then we must weigh the conjugacy classes of odd permutations in \eqref{w-mat} with $0$. For any $\lambda\vdash n$, we denote the transpose of $\lambda$ by $\lambda^\prime$. Using the fact that $\chi^{\lambda^{\prime}} = \chi^{[1^n]} \otimes \chi^\lambda$, we conclude that the eigenvalues of $A$ corresponding to $\chi^\lambda$ and $\chi^{\lambda^\prime}$ are equal.\label{rmk:d} 
			
			\item If \eqref{max}, \eqref{min}, and \eqref{others} are satisfied, then by the weighted Ratio bound we have 
			$$\alpha(\Gamma_{\sym(n)}) \leq \frac{n!}{1-\frac{\binom{n}{2}-1}{-1}} = 2(n-2)!.$$
			Moreover, if $Y$ is the component of $\Gamma$ that contains the identity element, then $\theta$ is an eigenvalue of $\Gamma$ with multiplicity $m$ if and only if $\theta$ is an eigenvalue of $Y$ with multiplicity $\frac{m}{2}$. Since the graph $Y$ is clearly a spanning subgraph of $\Gamma_{\alt(n)}$,  we also have that $$\alpha(\Gamma_{\alt(n)}) \leq \alpha(Y) \leq \frac{\frac{n!}{2}}{1-\frac{\binom{n}{2}-1}{-1}} = (n-2)!.$$
		\end{enumerate}\label{rmk:2-setwise}
	\end{rmk}
	
\noindent	Now, let us start with the proof. First, we give the values of some irreducible characters of relatively low degree on the conjugacy classes with cycle types appearing in \eqref{w-mat}.
		\begin{table}[H]
		\centering\small 
		\begin{tabular}{|c|c|c|c|c|c|c|} \hline
			& Cycle type & $[n]$  & $[n-1,1]$ & $[n-3,3]$ & $[n-4,3,1]$ & $[n-5,4,1]$\\
			& & & & & & \\
			Character  & Degree & & & & &\\ \hline
			$\chi^{[n]}$      & $1$ &  $1 $ & $1$ & $1$  & $1$ & $1$  \\ \hline
			$\chi^{[n-1,1]}$  & $n-1$ & $-1$ & $0$  & $-1$ & $0$ & $0$ \\ \hline
			$\chi^{[n-2,2]}$ & $\binom{n}{2}-n$ & $0$ &  $-1$ & $0$ & $-1$ & $-1$ \\ \hline
			$\chi^{[n-2,1^2]}$ & $\binom{n-1}{2}$ & $1$ &  $0$ & $1$ & $0$ & $0$ \\ \hline
			$\chi^{[1^{n}]}$ & $1$ & $(-1)^{n-1}$ &  $(-1)^n$ & $(-1)^{n}$ & $(-1)^{n-1}$ & $(-1)^{n-1}$ \\ \hline
		\end{tabular}
		\caption{Character values of low degree irreducible characters.}\label{tab:final-character-values}
	\end{table}
	
	The sizes of the conjugacy classes with cycle type $[n]$,   ${[n-1,1]}$,  ${[n-3,3]}$,  ${[n-4,3,1]}$, and  ${[n-5,4,1]}$ are respectively $(n-1)!$, $n(n-2)!$, $2\binom{n}{3}(n-4)!$, $8\binom{n}{4}(n-5)!$, and $30\binom{n}{5}(n-6)!$. Define %$$ 
	
	\begin{align*}
		\begin{array}{lllll}
			\omega_1 = (n-1)! x_{[n]}& & \omega_2 = n(n-2)!x_{[n-1,1]} & & \omega_3 = 2\binom{n}{3}(n-4)! x_{[n-3,3]}\\
			& & \\
			\omega_4 = 8\binom{n}{4}(n-5)!x_{[n-4,3,1]} & &  \omega_5 = 30\binom{n}{5}(n-6)!x_{[n-5,4,1]}.
		\end{array}
	\end{align*}
	
	By Lemma~\ref{lem:general-eigenvalue-formula-scheme}, the eigenvalue of $A$ afforded by the irreducible character corresponding to $\lambda \vdash n$ is
	\begin{align}
		{ \xi^\lambda = \frac{1}{f^\lambda} \left( \omega_1 \chi^\lambda_{(n)} + \omega_2\chi^\lambda_{[n-1,1]} +\omega_3\chi^\lambda_{[n-3,3]} +\omega_4 \chi^\lambda_{[n-4,3,1]} +\omega_5\chi^\lambda_{[n-5,4,1]} \right)}.\label{eq:form-of-general-eigenvalues-sym}
	\end{align}
	
	To guarantee that \eqref{max} and \eqref{min} are satisfied, we need to impose the following system of linear equations
	\begin{align}
		\left\{ 
		\begin{aligned}
			\omega_1 + \omega_2 + \omega_3 + \omega_4  + \omega_5 &= \alpha\\
			-\omega_1 - \omega_3 &= -\beta\\
			-\omega_2 - \omega_4 - \omega_5 &= -\gamma,
		\end{aligned}
		\right.\label{eq:solution}
	\end{align}
	where $\alpha = \binom{n}{2}-1, \beta = n-1$, and $\gamma = \binom{n}{2}- n$. It is straightforward to see that this system of linear equation has infinitely many solutions and has three free variables. A general solution to \eqref{eq:solution} is as follows
	\begin{align}\label{eq:solution-omega}
		\left\{
		\begin{aligned}
			\omega_1 &= \beta -t_1\\
			\omega_2 &= \alpha -\beta  -t_2 - t_3\\
			\omega_3 &= t_1\\
			\omega_4 &= t_2\\
			\omega_5 &= t_3,
		\end{aligned}
		\right. 
	\end{align}
	where $t_1,t_2,t_3 \in \mathbb{R}$. We use these values in \eqref{eq:solution-omega} to derive the weights on the matrix $A$ defined in \eqref{w-mat}.
	Note that the eigenvalues corresponding to \eqref{w-mat} are in function of the parameter $\mathbf{t} = (t_1,t_2,t_3) \in \mathbb{R}^3$.
	
	Using \cite[Lemma~3.2]{meagher20202}, the only irreducible characters of $\sym(n)$ of degree less than $2\binom{n+1}{2}$ are those corresponding to the partitions $[n], [1^n], [n-1,1], [2,1^{n-1}], [n-2,2],[3,1^{n-3}], [n-2,1^2]$, and $[2^2,1^{n-4}]$. 
	 The eigenvalues of the irreducible constituents of the permutation character are $\xi^{[n]}= \alpha \mbox{ and } \xi^{[n-1,1]} = \xi^{[n-2,2]} = -1$. Using the character values in Table~\ref{tab:final-character-values}, the eigenvalues of $A$ corresponding to the remaining characters with low degree are
	\begin{align}
		\begin{split}
			\xi^{[1^n]}(\mathbf{t}) &= (-1)^{n-1} \left( 2\beta  - \alpha -2t_1 +2t_2 + 2t_3 \right)\\
			\xi^{[2,1^{n-2}]}(\mathbf{t}) &= \frac{(-1)^n}{n-1} \left( \beta -2 t_1 \right)\\
			\xi^{[2^2,1^{n-4}]}(\mathbf{t}) &= \frac{(-1)^{n-1}}{\binom{n}{2}-n} \left( \alpha - \beta -2t_2 -2t_3 \right)\\
			\xi^{[n-2,1^2]}(\mathbf{t}) &= \frac{\beta}{\binom{n-1}{2}} \\
			\xi^{[3,1^{n-3}]}(\mathbf{t}) &= \frac{(-1)^{n-1}}{\binom{n-1}{2}} \left(\beta-2t_1  \right).
		\end{split}\label{eq:low-degree-evalues}
	\end{align}
	
	We would like to find the parameters $(\omega_i)_{i=1,2,3,4,5}$ so that the above eigenvalues \eqref{eq:low-degree-evalues} are all in the interval $\left[-1,\alpha\right)$. 
	We consider two cases depending on the parity of $n$.

	\subsection{Even case}
	In order to satisfy \eqref{max}(see Remark~\ref{rmk:2-setwise}), we only need the conjugacy classes of derangement consisting of even permutations. That is, we need to assume that $\omega_1 = \omega_4 = \omega_5= 0$, or equivalently 
	\begin{align*}
		&\omega_1(\mathbf{t}) = 0 
		&&\omega_2(\mathbf{t}) = \alpha - \beta\\
		& \omega_3(\mathbf{t}) = t_1 
		&& \omega_4(\mathbf{t}) = 0\\
		& \omega_5(\mathbf{t}) =  0.
	\end{align*}
	
	Let $\mathbf{t}:=(t_1,t_2,t_3) = (\beta,0,0)$. We note that $\xi^{[1^n]}(\mathbf{t}) = \alpha$. Using Remark~\ref{rmk:2-setwise}~\eqref{rmk:d}, it is easy to check that for this value of $\mathbf{t}$ the eigenvalues in \eqref{eq:low-degree-evalues} are all bounded below by $-1$ and are at most $\alpha$. Let us prove that every eigenvalue of $A$ afforded by an irreducible character of degree larger than $2\binom{n+1}{2}$ is in the interval $\left[-1,1\right]$.
	 For any $\lambda \vdash n$ such that $f^\lambda \geq 2\binom{n+1}{2}$, the eigenvalue $\xi^\lambda(\mathbf{t})$ in \eqref{eq:form-of-general-eigenvalues-sym} is such that
	\begin{align*}
		\tiny 
		\left|\xi^\lambda(\mathbf{t})\right| &\leq  
		\frac{1}{f^\lambda} \left( \left|\omega_2(\mathbf{t})\chi^\lambda_{[n-1,1]}\right| + \left|\omega_3(\mathbf{t})\chi^\lambda_{[n-3,3]}\right| \right)\\
		&\leq \frac{1}{2\binom{n+1}{2}} \left( \omega_2(\mathbf{t}) + \omega_3(\mathbf{t})    \right)  \hspace*{1cm}\mbox{ (see Lemma~\ref{lem:character-values})} \\
		&=  \frac{\alpha}{2\binom{n+1}{2}}  < 1.
	\end{align*}
	In other words, the eigenvalues from irreducible characters of dimension larger than $2\binom{n+1}{2}$ are all at least $-1$. 
	
	\subsection{Odd case}
	To satisfy \eqref{max} (see Remark~\ref{rmk:2-setwise}), we need $\omega_2 = \omega_3  = 0$, or equivalently, %.
	 \begin{align*}
	 	&\omega_1(\mathbf{t}) = \beta &
	 	&\omega_2(\mathbf{t}) = \alpha - \beta -t_2 - t_3  =0\\
	 	& \omega_3(\mathbf{t}) =  0 &
	 	& \omega_4(\mathbf{t}) = t_2\\
	 	& \omega_5(\mathbf{t}) = t_3.
	 \end{align*}

	Let $\mathbf{t} = (0,t_2,t_3)$ for some $t_2,t_3\geq 0$ such that $t_2 + t_3 = \gamma$. Here, we also have that $\xi^{[1^n]}(\mathbf{t}) = \alpha$. For these values of $\mathbf{t}$, the eigenvalues in \eqref{eq:low-degree-evalues} are bounded from below by $-1$ and are at most $\alpha$. As all the weights are non-negative, for any $\lambda \vdash n$ such that $f^\lambda \geq 2\binom{n+1}{2}$, we have
	\begin{align*}
		\left|\xi^\lambda (\mathbf{t}) \right| &\leq  
		\frac{1}{f^\lambda} \left( \left|\omega_1(\mathbf{t}) \chi^\lambda_{(n)}\right| +  \left|\omega_4(\mathbf{t}) \chi^\lambda_{[n-4,3,1]}\right|+  \left|\omega_5(\mathbf{t})\chi^\lambda_{[n-5,4,1]}\right| \right)\\
		&=  \frac{\alpha }{2\binom{n+1}{2}} \leq 1.
	\end{align*}
	From this, we conclude that all eigenvalues of $A$ are at least $-1$. 
	
	Therefore, for any $n\geq 16$ we found a weighted adjacency matrix of the derangement graph $\Gamma_{\sym(n)}$ for which \eqref{max}, \eqref{min} and \eqref{others} are satisfied.
	We conclude that any $2$-setwise intersecting set of $\sym(n)$ is of size at most $2(n-2)!$. Furthermore, any $2$-setwise intersecting set of $\alt(n)$ is of size at most $(n-2)!$.
	
	\section{$\psl(2,k^2)$ acting on $1$-spaces when $k$ is even}\label{sect:sln-even}
	In this section, we prove Theorem~\ref{thm:sl2} for $k$ even. That is, we show that the subgroup $\psl(2,k^2)=\sln(2,k^2)$ acting on $1$-spaces of $\mathbb{F}_k^4$ has intersection density equal to $1$. We recall that $K_1$ and $K_2$ are respectively the finite fields of order $k$ and $k^2$ defined in Section~\ref{sect:primitive}.
	
	\subsection{Conjugacy classes}
	The conjugacy classes of $\sln(2,k^2)$ are available in \cite{adams2002character}. Let us describe the matrices of $\sln(2,k^2)$ given by the embedding given in Section~\ref{sect:primitive}. There are four families of conjugacy classes in $\sln(2,k^2)$, which are given below.
	\begin{enumerate}[(i)]
		\item The conjugacy class consisting of the identity matrix $I_4 =I_2\otimes I_2$.\label{conjugacy:first}
		\item The class consisting of matrices conjugate to the matrix $\begin{bmatrix}
			1 & 1\\
			0 &1
		\end{bmatrix} \otimes I_2$.\label{second-class}
		\item The conjugacy classes consisting of matrices conjugate to a matrix of the form $\begin{bmatrix}
			A & 0 \\
			0 & A^{-1}
		\end{bmatrix}$, where $A\in K_2 \setminus \{ I_2\}$. Note that this conjugacy class is equal to the conjugacy class of $\operatorname{Diag}(A^{-1},A)$.\label{third-class}
		\item \label{conjugacy:last} Consider the finite field $F\cong \mathbb{F}_{k^4}$ which is an extension of degree $2$ of the field $K_2$.
		 Let $N: F^* \to K_2^*$ be such that $N(Z) = ZZ^{k^2}$, for $Z\in F^*$. The map $N$ is a homomorphism of groups and $E := \ker N$ is a cyclic subgroup of $F^*$ of order $k^2+1$. The last type of conjugacy classes of $\sln(2,k^2)$ are obtained from the matrices in $E$. A conjugacy class of this type consists of matrices conjugate to a matrix of the form 
		\begin{align*}
			\mathcal{A}_Z :=\begin{bmatrix}
				0 & I_2\\
				I_2 & Z+Z^{k^2}
			\end{bmatrix}, 
		\end{align*}
		where $Z  \in E$. 
	\end{enumerate}
	
	\subsection{Derangements}
	
	Using this information about the conjugacy classes, we can easily describe the conjugacy classes of derangements of $\sln(2,k^2)$. There are two types of conjugacy classes of derangements in $\sln(2,k^2)$, which are described as follows.
	\begin{itemize}
		\item {\bf Type~1:} The conjugacy classes of matrices conjugate to a matrix the form $ 
		B = \begin{bmatrix}
			A & 0 \\
			0 & A^{-1}
		\end{bmatrix}$, where $A \in K_2 \setminus K_1$. 
		
		It is clear that if $A \in K_1$, then $B$ fixes the $1$-dimensional subspace spanned by 
		\begin{align*}
			\begin{bmatrix}
				1 &
				1 &
				0 &
				0
			\end{bmatrix}^T.
		\end{align*}
		By distinguishing a couple of cases, it is easy to see that these are the only conjugacy classes that admit fixed points in this class. 	 There are $\frac{k-2}{2}$ elements of $K_1$ that give such matrices in $\sln(2,k^2)$ (i.e., fixing a $1$-space). Therefore, there are $\frac{k^2-2}{2} - \frac{k-2}{2} = \binom{k}{2}$ conjugacy classes of derangements of this type. Each conjugacy class of this type has size $k^2(k^2+1)$.
		
		For any $A\in K_2^*$, we define
		\begin{align*}
			T_A&:=
			\begin{bmatrix}
				A & 0\\
				0 & A^{-1}
			\end{bmatrix}.
		\end{align*}
		We will denote the $\binom{k}{2}$ representatives of the conjugacy classes of derangements of this type by $T_{A_1},T_{A_2},\ldots,T_{A_{\binom{k}{2}}}$.
		
		\item {\bf Type~2:} The conjugacy classes of matrices of $\sln(2,k^2)$ that are described in \eqref{conjugacy:last}. There are $\frac{k^2}{2}$ classes of this type, each of size $k^2(k^2-1)$. We will show that these classes consist of derangements.
		
		Let us prove that $\mathcal{A}_Z = \begin{bmatrix}
			0 & I_2\\
			I_2 & Z+Z^{k^2}
		\end{bmatrix}$, for $Z \in E\setminus \{ I_2\}$, is a derangement for the action of $\sln(2,k^2)$ on the $1$-spaces of $\mathbb{F}_k^4$. 
		If $\mathcal{A}_Z$ is not a derangement, then it fixes a $1$-space of  $\mathbb{F}_{k}^4$ which is equivalent to saying that $\mathcal{A}_Z$ admits an eigenvalue in $\mathbb{F}_{k}$. Since the characteristic polynomial of $\mathcal{A}_Z$ over $\mathbb{F}_k$ is equal to $\det\left(	t^2I_2 +t \left(Z+Z^{k^2} \right) I_2 + I_2 \right) = \det \left( t I_2 + Z \right)\det\left( t I_2 + Z^{-1} \right)$ and $Z,Z^{-1} \in E \setminus \{I_2\}$ do not have any eigenvalues in $\mathbb{F}_k$, we conclude that $\mathcal{A}_Z$ does not have an eigenvalue in $ \mathbb{F}_k$.
		
		We will denote the $\frac{k^2}{2}$ representatives of the conjugacy classes of derangements of this type by $\mathcal{A}_{Z_1},\mathcal{A}_{Z_2},\ldots ,\mathcal{A}_{Z_{\frac{k^2}{2}}}$.
	\end{itemize}
	
	\subsection{The irreducible characters}
	In this subsection, we describe the irreducible characters and certain (not all) constituents of the permutation character of the action of $\sln(2,k^2)$ on the $1$-spaces of $\mathbb{F}_k^4$.
	
	The irreducible characters of $\sln(2,k^2)$ are available in \cite{adams2002character}. We refer to \cite[Page~5,10]{adams2002character} for the notations and the properties of the irreducible representations. There are four families of irreducible characters of $\sln(2,k^2)$, when $k$ is power of $2$. These characters are:
	\begin{itemize}
		\item the trivial character $\rho^\prime (1)$,
		\item the character $\overline{\rho}(1)$,
		\item the character $\rho(\alpha)$, where $\alpha$ is an irreducible character of $K_2^*$ which is not the trivial character,
		\item the character $\pi(\chi)$, where $\chi$ is a non-trivial irreducible character of the kernel $E$ (a cyclic group of order $k^2+1$) of the norm map $N$ defined in the previous subsection. 
	\end{itemize} 
	
	Recall that the permutation character of a permutation group $G$ is the character denoted by $\fix$ which gives the number of fixed points of a permutation of $G$.
	\begin{prop}
		The irreducible characters $\rho^\prime(1)$ and $\overline{\rho}(1)$ are constituents of the permutation character of the action of $\sln(2,k^2)$ on the $1$-spaces of $\mathbb{F}_k^4$.\label{prop:perm-char}
	\end{prop}
	\begin{proof}
		Since a character of a group is a non-negative linear combination of irreducible characters, it is enough to prove that the coefficients of $\rho^\prime(1)$ and $\overline{\rho}(1)$ in this combination are both nonzero. It is obvious that the trivial character $\rho^\prime(1)$ is a constituent of the permutation character since $\langle \rho^\prime(1),\fix \rangle = \frac{1}{|\sln(2,k^2)|}\sum_{g\in \sln(2,k^2)}\fix(g) = 1$ ($\sln(2,k^2)$ acts transitively on the $1$-spaces).
		
		We also prove that $\langle \fix,\overline{\rho}(1) \rangle =1$. If $A \in K_1 \setminus \{I_2\}$, then the matrix $\operatorname{Diag}(A,A^{-1})$ fixes $2(k+1)$ elements.
		Now using the character table of $\sln(2,k^2)$, we have 
		\begin{align*}
			\langle \overline{\rho}(1),\fix \rangle 
			&= \frac{1}{|\sln(2,k^2)|} \sum_{A\in \sln(2,k^2)} \overline{\rho}(1)(A)\ \overline{\fix(A)}\\
			&=\frac{1}{k^2(k^4-1)}\left(k^2  \times 1 \times (k^2+1)(k+1) + 0+1\times  \frac{(k-2)}{2} k^2(k^2+1) \times 2(k+1)  \right. \\  & \left. \ \ \ + 1\times  \binom{k}{2}k^2(k^2+1)\times 0  +(-1) \times \frac{k^4(k^2-1)}{2}\times 0\right)\\
			&= \frac{1}{k^2(k^4-1)}\left(k^2 (k^2+1)(k+1) + k^2(k+1)(k^2+1)(k-2)\right)\\
			&= \frac{1}{k^2(k^4-1)}\left(k^2 (k^2+1)(k+1) (k-1)\right)\\
			&= \frac{1}{k^2(k^4-1)}\left(k^{2}(k^4-1)\right)\\
			&=1.
		\end{align*} 
		This completes the proof.
	\end{proof}
	
	Next, we describe the values of the irreducible characters $\rho(\alpha)$ and $\pi(\chi)$ where $\alpha$ is an irreducible representation of $K_2^*$ and $\chi$ is an irreducible representation of $E$. 
	
	Recall that  $\mathcal{A}_{Z_1},\mathcal{A}_{Z_2},\ldots,\mathcal{A}_{Z_{\frac{k^2}{2}}}$ are the representatives of the $\frac{k^2}{2}$ conjugacy classes of derangements of Type~2 in $\sln(2,k^2)$.
	\begin{lem}
		Let $\chi$ be a non-trivial irreducible representation of $E$. Then, we have
		\begin{align}
			\sum_{i=1}^{\frac{k^2}{2}}\pi(\chi) (\mathcal{A}_{Z_i}) &=  1.\label{eq:character4-property}
		\end{align} \label{lem:pi}
	\end{lem}
	\begin{proof}
		Using the character table of $\sln(2,k^2)$, we have $\pi(\chi)(\mathcal{A}_{Z_i}) = -\left(\chi(Z_i) + \chi(Z_i^{-1})\right)$, for $i\in \{ 1,2,\ldots,\frac{k^2}{2} \}$. By noting that $\chi$ is an irreducible character of the cyclic group $E$ of order $k^2+1$ and for $i,j\in \{1,2,\ldots,\frac{k^2}{2}\}$, the equality $\chi(Z_i) = \chi(Z_j)$ holds if and only if $i=j$, we have
		\begin{align*}
			\sum_{i=1}^{\frac{k^2}{2}}\pi(\chi) (\mathcal{A}_{Z_i}) &= -\sum_{i=1}^{\frac{k^2}{2}}\left(\chi(Z_i)+\chi(Z_i^{-1})\right) = -\sum_{\zeta \in \left\{z\in \mathbb{C} \mid z^{k^2+1} = 1,\ z\neq 1 \right\}} \zeta =  1.\qedhere%\label{eq:character4-property}
		\end{align*} 
	\end{proof}
	
	 Recall that $T_{A_1}$, $T_{A_2},\ldots, T_{A_{\binom{k}{2}}}$ are representatives of the conjugacy classes of derangements of type~1.
	\begin{lem}
		Let $\alpha$ be a non-trivial irreducible character of $K_2^*$ and let $\alpha_{|_{K_1^*}}$ be the restriction of $\alpha$ on $K_1^*$. We have
		\begin{align}
			\sum_{i=1}^{\binom{k}{2}}\rho(\alpha)(T_{A_i})  =
			\left\{
			\begin{aligned}
				&-(k-1), & \mbox{ if $\alpha_{|_{K_1^*}}$ is the trivial character of $K_1^*$};\\
				&0, & \mbox{ otherwise.} 
			\end{aligned}
			\right.\label{eq:character3-property}
		\end{align}\label{lem:alpha}
	\end{lem}
	\begin{proof}
		Let $1_{K_1^*}$ be the trivial character of $K_{1}^*$. Let $\alpha$ be an irreducible representation of $K_2^*$. Let $T_{x_1I_2},T_{x_2I_2},\ldots,T_{x_{\frac{k-2}{2}}I_2}$ be the representatives of the conjugacy classes of non-derangements in \eqref{third-class}. By the property of $(k^2-1)$-th roots of unity, we have
		\begin{align*}
			\sum_{i=1}^{\binom{k}{2}}\rho(\alpha)(T_{A_i}) + \sum_{j=1}^{\frac{k-2}{2}}\rho(\alpha)(T_{x_jI_2}) = \sum_{A\in K_2^*\setminus\{I_2\}} \alpha(A)= -1. 
		\end{align*}
		Assume that the restriction $\alpha_{|_{K_1^*}}$ of $\alpha$ on $K_1^*$ is equal to the trivial representation $1_{K_1^*}$ of $K_1^*$. Then, we have 
		\begin{align*}
			\sum_{i=1}^{\binom{k}{2}}\rho(\alpha)(T_{A_i})  = -1- \sum_{j=1}^{\frac{k-2}{2}}\rho(\alpha)(T_{x_jI_2}) = -1- \sum_{j=1}^{\frac{k-2}{2}} (1+1) = -1-(k-2) = -(k-1).
		\end{align*}
		
		If $\alpha_{|_{K_1^*}} \neq 1_{K_1^*}$, then by the orthogonality of characters and the fact that the degree of $\alpha$ is $1$, we have 
		\begin{align}
			\left\langle 1_{K_1^*},\alpha_{|_{K_1^*}} \right\rangle_{K_1^*} = \frac{1}{|K_1^*|} \sum_{x\in  K_1^*} \alpha_{|_{K_1^*}}(x) = 0.\label{eq:orthogonality}
		\end{align} 
		Using \eqref{eq:orthogonality},	we deduce that $\displaystyle \sum_{j=1}^{\frac{k-2}{2}}\rho(\alpha)(T_{x_jI_2}) = \sum_{A\in K_1^* \setminus \{I_2\}} \alpha(A) =  -1$. Hence, 
		\begin{equation*}
			\sum_{i=1}^{\binom{k}{2}}\rho(\alpha)(T_{A_i}) =-1-\sum_{j=1}^{\frac{k-2}{2}} \rho(T_{x_jI_2}) = 0.\qedhere
		\end{equation*}
	\end{proof}

	\subsection{Maximum cocliques}
	In this subsection, we prove that $\sln(2,k^2)$ acting on the $1$-spaces of $\mathbb{F}_{k}^4$ has the EKR property, whenever $k$ is even. To do this, we will use a weighted adjacency matrix for which the weighted Ratio Bound yields the order of a point-stabilizer.
	
	For any $i\in \{1,2,\ldots,\binom{k}{2}\}$, we let $\mathcal{T}^{(1)}_{A_i}$ be the matrix in the conjugacy class scheme $\mathcal{A}(\sln(2,k^2))$ that corresponds to the conjugacy class of $T_{A_i}$. For a conjugacy class of Type~2, we let $\mathcal{T}^{(2)}_{Z}$ be the matrix in $\mathcal{A}(\sln(2,k^2))$ that is obtained from the conjugacy class of the matrix $\mathcal{A}_Z = \begin{bmatrix}
		0 & I_2\\
		I_2 & Z+Z^{k^2}
	\end{bmatrix}$, where $Z \in E$. 
	
	Let us uniformly assign weights to the conjugacy classes of derangements of the same type. That is, let $\omega_1$ and $\omega_2$ be two real numbers and define the weighted adjacency matrix
	\begin{align}
		\mathcal{T}(\omega_1,\omega_2) &= \omega_1 \sum_{i=1}^{\binom{k}{2}} \mathcal{T}^{(1)}_{A_i}
		+
		\omega_2
		\sum_{j=1}^{\frac{z^2}{2}}\mathcal{T}_{Z_i}^{(2)}.\label{eq:weighted-adjacency_matrix-psl}
	\end{align} 
	
	Now, we use Lemma~\ref{lem:general-eigenvalues-association-scheme} to find the eigenvalues of $\mathcal{T}=\mathcal{T}(\omega_1,\omega_2)$ as a function of $\omega_1$ and $\omega_2$. 
	We would like to find $\omega_1$ and $\omega_2$ such that 
	\begin{itemize}
		\item The eigenvalue of $\mathcal{T}(\omega_1,\omega_2)$ afforded by $\rho^\prime(1)$ is $(k^2+1)(k+1)-1$.
		\item The eigenvalue of $\mathcal{T}(\omega_1,\omega_2)$ afforded by $\overline{\rho}(1)$ is $-1$.
		\item The eigenvalues of $\mathcal{T}(\omega_1,\omega_2)$ afforded by all other irreducible characters are in the interval $\left[-1,(k^2+1)(k+1)-1\right]$.
	\end{itemize}
	
	\begin{table}[H]
		\begin{tabular}{|c|c|c|c|} \hline
			& & Type~1 : $\displaystyle\sum_{i=1}^{\binom{k}{2}} \mathcal{T}^{(1)}_{A_i}$  & Type~2: $\displaystyle\sum_{j=1}^{\frac{k^2}{2}}\mathcal{T}_{Z_j}^{(2)}$ \\
			%\hline
			Character & Degree & & \\ \hline
			$\rho^\prime(1)$ & $1$    &  $\binom{k}{2} k^2(k^2+1) $ & $\frac{k^4}{2}(k^2-1)$ \\ \hline
			$\overline{\rho}(1)$  & $k^2$ & $\binom{k}{2} (k^2+1) $ & $-\frac{k^4}{2}(k^2-1)$  \\ \hline
		\end{tabular}
		\caption{Character values afforded by $\rho^\prime(1)$ and $\overline{\rho}(1)$.}\label{tab:eigenvalues}
	\end{table}
	
	Using Table~\ref{tab:eigenvalues}, we need the eigenvalues of $\mathcal{T}(\omega_1,\omega_2)$ afforded by $\rho^\prime(1)$ and $\overline{\rho}(1)$ to satisfy the system of linear equations
	\begin{align}
		\left\{
		\begin{aligned}
			\binom{k}{2} k^2(k^2+1) \omega_1 + \frac{k^4}{2}(k^2-1)\omega_2 &= (k^2+1)(k+1)-1\\
			\frac{1}{k^2}\left(\binom{k}{2} k^2(k^2+1)\omega_1 -\frac{k^4}{2}(k^2-1)\omega_2 \right) &= -1.
		\end{aligned}
		\right.\label{eq:weighted-eigenvalue-psl}
	\end{align}
	The system of linear equations \eqref{eq:weighted-eigenvalue-psl} with indeterminate $\omega_1$ and $\omega_2$ has a unique solution, which is 
	\begin{align}
		\begin{split}
			\omega_1 &= \frac{k^3+k}{2\binom{k}{2} k^2(k^2+1)}\\
			\omega_2 &= \frac{k^3+2k^2+k}{k^4(k^2-1)}.
		\end{split}\label{eq:solutions-psl}
	\end{align}
	Now, let us apply the values in \eqref{eq:solutions-psl} to the weighted adjacency matrix given in \eqref{eq:weighted-adjacency_matrix-psl}. Using Lemma~\ref{lem:pi}, the eigenvalue of $\mathcal{T}(\omega_1,\omega_2)$ afforded by the irreducible character $\pi(\chi)$, where $\chi$ is a non-trivial irreducible character of $E$, is 
	\begin{align*}
		\frac{k^3+2k^2+k}{k^4(k^2-1)} \times  \frac{1}{k^2-1} \sum_{i=1}^{\frac{k^2}{2}} k^2(k^2-1) \pi(\chi) (\mathcal{A}_{Z_i}) = \frac{k+1}{k(k-1)}.
	\end{align*}
	By Lemma~\ref{lem:alpha}, the eigenvalue of $\mathcal{T}(\omega_1,\omega_2)$ afforded by the irreducible character $\rho(\alpha)$, where $\alpha$ is a non-trivial irreducible character of $K_2^*$, is either equal to $0$ or
	\begin{align*}
		\frac{k^3+k}{2\binom{k}{2}k^2(k^2+1)}\times \frac{1}{k^2+1} \left( k^2(k^2+1) \times \left(-(k-1)\right)  \right) = -1,
	\end{align*}
	depending on the restriction of $\alpha$ on $K_1^*$ (see \eqref{eq:character3-property}). Therefore, the eigenvalues of $\mathcal{T}(\omega_1,\omega_2)$ are
	\begin{align*}
		(k^2+1)(k+1)-1,\  \frac{k+1}{k(k-1)},\  0,\ -1.
	\end{align*}
	By Lemma~\ref{lem:weighted-ratio-bound}, we have 
	\begin{align*}
		\alpha(\Gamma_{\sln(2,k^2)}) \leq \frac{|\sln(2,k^2)|}{1-\frac{(k^2+1)(k+1)-1}{-1}} = \frac{|\sln(2,k^2)|}{(k^2+1)(k+1)}.
	\end{align*}
	In other words, $\rho(\sln(2,k^2)) = 1$. 
	
	\section{$\psl(2,k^2)$ acting on the $1$-spaces when $k$ is odd}\label{sect:sln-odd}
	
	In this section, we prove Theorem~\ref{thm:sl2} when $k$ is an odd prime power. We observe from the character table in \cite{adams2002character} that the conjugacy classes of derangements of the action of $\psl(2,k^2)$ on the $1$-spaces are similar to the ones in the case where $k$ is even. They are one of the following.
	\begin{itemize}
		\item {\bf Type~1:} The matrices conjugate to $T_{A}:= \begin{bmatrix}
			A & 0 \\
			0 & A^{-1}
		\end{bmatrix}$, 
		where $A \in K_2^* \setminus K_1^*$. Note that $T_A,\ T_{-A},\ T_{A^{-1}}, \mbox{ and } T_{-A^{-1}}$ are in the same conjugacy class for any $A \in K_2^* \setminus \{ \pm I_2 \}$. Each conjugacy class of this type has $k^2(k^2+1)$ elements and there are $\frac{k^2-5}{4}-\frac{k-5}{4} = \frac{1}{2}\binom{k}{2}$ conjugacy classes of this type.
		
		Let $T_{A_1},T_{A_2},\ldots,T_{A_{\frac{1}{2}\binom{k}{2}}}$ be distinct representatives of the conjugacy classes of derangements of this type.
		\item {\bf Type~2:} The matrices conjugate to $\mathcal{A}_Z:=\begin{bmatrix}
			0 & I_2\\
			-I_2 & Z + Z^{k^2}
		\end{bmatrix}$, where $Z\in E \setminus \{ I_2,-I_2 \}$ ($E$ is the kernel of the norm map $N$). Note that $\mathcal{A}_{Z},\ \mathcal{A}_{Z^{-1}},\ \mathcal{A}_{-Z}, \mbox{ and } \mathcal{A}_{-Z^{-1}}$ are in the same conjugacy class, for any $Z\in E\setminus \{ \pm I_2 \}$. Each conjugacy class has size $k^2(k^2-1)$ and there are $\frac{k^2-1}{4}$ conjugacy classes in total.
	
		Let $\mathcal{A}_{Z_1}, \mathcal{A}_{Z_2}, \ldots,\mathcal{A}_{Z_{\frac{k^2-1}{4}}}$ be distinct representatives of the conjugacy classes of derangements of this type.
	\end{itemize}
	
	The irreducible characters of $\psl(2,k^2)$ for $k$ odd are available in \cite{adams2002character}. 
	 We keep the same notations for the irreducible characters from \cite{adams2002character}. 
	The irreducible characters of $\psl(2,k^2)$ for $k$ odd are:
	\begin{itemize}
		\item the trivial character $\rho^\prime (1)$,
		\item the character $\overline{\rho}(1)$,
		\item the character $\rho(\alpha)$, where $\alpha$ is a non-trivial irreducible character of $K_2^*$ such that $\alpha^2$ is not the trivial character (note that $\rho(\alpha) \cong \rho(\alpha^{-1})$),
		\item the character $\pi(\chi)$, where $\chi$ is a non-trivial irreducible character of the kernel $E = \ker N$ such that $\chi^2 \neq 1$ and $\chi \neq \overline{\chi}$. Note that $\pi(\chi) \cong \pi(\overline{\chi})$.
		\item The characters $\omega_e^{+}$ and $\omega_e^{-}$.
	\end{itemize}
	
	The proof of the next result is similar to that of Proposition~\ref{prop:perm-char}, so we omit it. 
	\begin{prop}
		The irreducible characters $\rho^\prime(1)$ and $\overline{\rho}(1)$ are constituents of the permutation character of $\sln(2,k^2)$ acting on the $1$-spaces.
	\end{prop}
	
	The proofs of the next two lemmas are also similar to their analogue in the previous section. 		Recall that $\mathcal{A}_{Z_1}, \mathcal{A}_{Z_2}, \ldots,\mathcal{A}_{Z_{\frac{k^2-1}{4}}}$ are representatives of the conjugacy classes of derangements of this type~2.
		\begin{lem}
		Let $\chi$ be a non-trivial irreducible representation of $E$ such that $\chi^2 \neq 1$ and $\chi \neq \overline{\chi}$. Then, we have
		\begin{align}
			\sum_{i=1}^{\frac{k^2-1}{4}}\pi(\chi) (\mathcal{A}_{Z_i}) &=  -1.\label{eq:character4-property2}
		\end{align} \label{lem:pi2}
	\end{lem}
	\begin{proof}
		Note that $\sum_{Z\in E} \chi(Z) =0$, since $\chi$ is a non-trivial irreducible character of the cyclic group $E$. As $\chi(-I_2)^2 = \chi(I_2) =  1$, we know that $\chi(-I_2) = \pm 1$. If $\chi (-I_2) = - 1$, then we note that for any $Z \in E \setminus \{ \pm I_2\}$, we have
		\begin{align*}
			\chi(-Z)  + \chi(-Z^{-1}) = \chi(-I_2) \chi(Z) + \chi(- I_2) \chi(Z^{-1}) =  - \chi(Z) - \chi(Z^{-1}).
		\end{align*}
		 Using the fact that $\mathcal{A}_{Z},\ \mathcal{A}_{Z^{-1}},\ \mathcal{A}_{-Z}, \mbox{ and } \mathcal{A}_{-Z^{-1}}$ are in the same conjugacy class for any $Z\in E \setminus \{ \pm I_2 \}$, we have $\pi(\chi) \left( \mathcal{A}_{Z} \right) = \pi (\chi) \left( \mathcal{A}_{-Z} \right)$ and so
		 \begin{align*}
		 	\chi(Z) + \chi(Z^{-1}) = \chi(-Z) + \chi(-Z^{-1}) = -\chi(Z) - \chi(Z^{-1}).
		 \end{align*}
	 	Consequently, we have $\chi(Z) + \chi(Z^{-1}) = 0$, which implies that $\chi(Z) = 0$. However, the latter is impossible since $\chi$ has degree $1$, which means that $\chi(Z)$ is a root of unity. Therefore, we conclude that $\chi(I_2) = 1$.
	 	
	 	Since $\mathcal{A}_{Z},\ \mathcal{A}_{Z^{-1}},\ \mathcal{A}_{-Z}, \mbox{ and } \mathcal{A}_{-Z^{-1}}$ are in the same conjugacy class for any $Z\in E \setminus \{ \pm I_2 \}$, we have
		\begin{align*}
			\sum_{Z\in E} \chi(Z) &= \chi(I_2) + \chi(-I_2) + \sum_{Z\in E \setminus \{ \pm I_2 \}}  \chi(Z)\\
			 &=  \chi(I_2) + \chi(-I_2) + \sum_{i=1}^{\frac{k^2-1}{4}} \left(\chi(Z_i) + \chi(Z_i^{-1})\right) + \sum_{i=1}^{\frac{k^2-1}{4}} \left( \chi(-Z_i) + \chi(-Z_i^{-1}) \right)\\
			 &=2 + 2 \sum_{i=1}^{\frac{k^2-1}{4}} \left( \chi(Z_i) + \chi(Z_i^{-1}) \right) = 0.
		\end{align*}
		In other words, $\sum_{i=1}^{\frac{k^2-1}{4}} \left( \chi(Z_i) + \chi(Z_i^{-1}) \right) = -1$.
		
	\end{proof}
	Recall that $T_{A_1}$, $T_{A_2},\ldots, T_{A_{\binom{k}{2}}}$ are representatives of the conjugacy classes of derangements of type~1. We state the next lemma without proof since it is similar to what we saw in the previous section.
	\begin{lem}
		Let $\alpha$ be a non-trivial irreducible character of $K_2^*$ and let $\alpha_{|_{K_1^*}}$ be the restriction of $\alpha$ on $K_1^*$. We have
		\begin{align}
			\sum_{i=1}^{\frac{1}{2}\binom{k}{2}}\rho(\alpha)(T_{A_i})  =
			\left\{
			\begin{aligned}
				&-\frac{1}{2}(k-1), & \mbox{ if $\alpha_{|_{K_1^*}}$ is the trivial character of $K_1^*$};\\
				&0, & \mbox{ otherwise.} 
			\end{aligned}
			\right.\label{eq:character3-property2}
		\end{align}\label{lem:alpha2}
	Moreover, if $\zeta$ is the unique non-trivial irreducible character of $K_2^*$ such that $\zeta^2$ is the trivial character, then	we have 
	\begin{align}
		\sum_{i=1}^{\frac{1}{2}\binom{k}{2}} \omega_e^{\pm}(T_{A_i}) = \sum_{i=1}^{\frac{1}{2}\binom{k}{2}} \zeta(A_i) = - \frac{k-1}{4}.
	\end{align}
	\end{lem}
	
	Now, we prove that there exists a weighted adjacency matrix for which the weighted Ratio Bound yields the order of a point-stabilizer. For any $i\in \{ 1,2,\ldots,\frac{1}{2} \binom{k}{2} \}$, we let $\mathcal{T}_{A_i}^{(1)}$ be the matrix in the conjugacy class scheme $\psl(2,k^2)$ corresponding to the conjugacy class of $T_{A_i}$. Similarly, for $j \in \{ 1,2,\ldots,\frac{k^2-1}{4} \}$, we let $\mathcal{T}_{Z_i}^{(2)}$ be the matrix of the conjugacy class scheme of $\psl(2,k^2)$ corresponding to the conjugacy class of $\mathcal{A}_{Z_i}$.
	
	 Now, consider the weighted adjacency matrix
	\begin{align*}
		\mathcal{T}(\omega_1,\omega_2) &= \omega_1 \sum_{i=1}^{\frac{1}{2}\binom{k}{2}} \mathcal{T}^{(1)}_{A_i}
		+
		\omega_2
		\sum_{j=1}^{\frac{k^2-1}{2}}\mathcal{T}_{Z_i}^{(2)}.
	\end{align*}
	For the values 
	\begin{align*}
		\omega_1 &= \frac{k(k^2+1)}{\binom{k}{2}k^2(k^2+1)} \mbox{ and }
		\omega_2 = \frac{2k(k+1)^2}{k^2(k^2-1)^2},
	\end{align*}
	the eigenvalues of $\mathcal{T}(\omega_1,\omega_2)$ afforded by the irreducible characters $\rho^\prime(1)$ and $\overline{\rho}(1)$ are $(k+1)(k^2+1)-1$ and $-1$, respectively. The other eigenvalues of $\mathcal{T}(\omega_1,\omega_2)$ are
	\begin{align*}
		-1,\ - \frac{k(k^2+1)}{2k(k^2-1)(k+1)}, \ \mbox{ and } 0.
	\end{align*}
	Using the weighted Ratio Bound, we have $\alpha(\Gamma_{\psl(2,k^2)})\leq \frac{|\psl(2,k^2)|}{{1 -\frac{(k^2+1)(k+1)-1}{-1}}} = \frac{|\psl(2,k^2)|}{(k^2+1)(k+1)}$. Hence, \\ $\rho(\psl(2,k^2)) = 1$.
	\section{Future work}\label{sect:future-works}
	\subsection{Summary of the results}
	In this paper it is proved that if $G\leq \sym(\Omega)$ is imprimitive of degree $pq$ with at least two systems of imprimitivity, then $G$ has the EKR property. Moreover, if $G\leq \sym(\Omega)$ is primitive of degree $pq$, with socle equal to one of the groups in lines 1-11, 14, 16 and 17 in Table~\ref{table:classification}, then $G$ has the EKR property. In order to prove the latter, we used Corollary~\ref{cor:no-hom} which says that if $\soc(G)$ (which is transitive) has the EKR property, then so does $G$. 
	
	\begin{table}[H]
		\centering
		\begin{tabular}{|c|c|c|c|}
			\hline
			 $\soc(G)$ & $(p,q)$ & action & Information \\
			\hline\hline
			 $\po (2d,2)$ & $\left(2^d - \varepsilon,2^{d-1} + \varepsilon\right)$ & $\begin{aligned}
				\mbox{singular}\\
				\mbox{ 1-spaces}
			\end{aligned}$ & $\begin{aligned}
				&\varepsilon = 1 \mbox{ and $d$ is a Fermat prime} \\ 
				&\varepsilon = -1 \mbox{ and $d-1$ is a Mersenne prime}
			\end{aligned}$\\
			\hline
			 $\psl(2,p)$ & $\left( p,\frac{p+1}{2} \right)$ & cosets of $D_{p-1}$ & $p\geq 13$ and $p\equiv 1 (\operatorname{mod}\ 4)$\\
			\hline
			 $\psl(2,q^2)$ & $\left( \frac{q^2+1}{2},q \right)$ & cosets of $\pgl(2,q)$ & \\
			\hline
			 $\psl(2,61)$ & $ (61,31)$ & cosets of $\alt(5)$& \\
			\hline
		\end{tabular}
		\caption{The remaining simply primitive groups of degree $pq$.}\label{table:status}
	\end{table}
	
	The remaining families of primitive groups of degree $pq$ that are left to check are those in Table~\ref{table:status}. These groups are the socles of primitive groups of degree $pq$ that do not admit imprimitive subgroups (see \cite[Table~3]{du2018hamilton}).  We make the following conjecture about these groups.
	\begin{conj}
		The groups in Table~\ref{table:status} have the EKR property.\label{conj:final-conj}
	\end{conj}
	We note that the technique used to prove the main results in this work (i.e., the weighted Ratio Bound) does not work for the groups in Conjecture~\ref{conj:final-conj} in general.
	
	Now, we give some other interesting directions for future research.

	\subsection{$\psl(2,q)$ acting on $2$-subsets of $\pg(1,q)$}
	In Section~\ref{sect:EKR-for-sym}, Section~\ref{sect:sln-even}, Section~\ref{sect:sln-odd}, we found the intersection density of groups that are more general than what is needed (see Table~\ref{table:classification}) to prove Conjecture~\ref{conj-main}~\eqref{conj3} for primitive groups. 
	
	Another interesting problem is to consider the intersection density of the group $\psl(2,q)$, where $q$ a prime power, acting on $2$-subsets of $\pg(1,q)$. Note that when $q =1\mod 4$ is a prime, then this group action is permutation isomorphic to the group in line~13 of Table~\ref{table:classification}. 
	
	Now, we give a short discussion on the intersection density of these groups depending on the congruence class of $p$ modulo $4$.
	\subsubsection{When $q$ is even}
	When $q$ is an even prime power, then $\psl(2,q) = \sln(2,q)$. The action of $\psl(2,q)$ considered in this subsection is the one on the $2$-subsets of $\pg(1,q)$. We prove that this group does not have the EKR property.
	\begin{thm}
		Let $q = 2^k$, for some $k\geq 1$. If $\mathcal{F} \subset \psl(2,q)$ is intersecting, then $|\mathcal{F}| \leq q(q-1)$. Moreover, $\rho(\psl(2,q)) = \frac{q}{2}$.\label{thm:psl-even}
	\end{thm}
	
	Again, we use the character table of $\psl(2,q)$ in \cite{adams2002character}. We use the notation in \cite{adams2002character} for the irreducible characters and the conjugacy classes. Using Lemma~\ref{lem:eig-norm} and the character table in \cite{adams2002character}, it is not hard to see that the eigenvalues (with the multiplicities) of $\psl(2,q)$ acting on the $2$-subsets of $\pg(1,q)$ are
	\begin{align*}
		\frac{q^2(q-1)}{2}^{(1)}, q^{\left(\frac{q(q-1)^2}{2}\right)}, 0^{\left(\frac{(q+1)^2(q-2)}{2}\right)}, -\frac{q(q-1)}{2}^{(q^2)}.
	\end{align*}
	Using the Ratio Bound, we get that
	\begin{align*}
		\alpha (\Gamma_{\psl(2,q)}) \leq \frac{-\frac{q(q-1)}{2}}{-\frac{q(q-1)}{2} - \frac{q^2(q-1)}{2} }q(q-1)(q+1) = q(q-1).
	\end{align*}
	
	In other words, if $\mathcal{F} \subset \psl(2,q)$ acting on the $2$-subsets of $\pg(1,q)$ is intersecting, then $|\mathcal{F}| \leq q(q-1)$. If $\beta \in \mathbb{F}_q$ is a primitive element, then this upper bound is attained by the intersecting set consisting of all matrices of the form
	\begin{align*}
		\begin{bmatrix}
			\beta^t & \beta^{-t}y\\
			0& \beta^{-t}
		\end{bmatrix},
	\end{align*}
	where $y\in \mathbb{F}_q$ and $t\in \{1,\ldots,q-1\}$. Note that these matrices form a subgroup and any such matrix has two distinct eigenvalues, so each of them fixes a $2$-subset of $\pg(1,q)$.
	
	 We conclude that $\rho(\psl(2,q)) = \frac{q(q-1)}{2(q-1)} = \frac{q}{2}$.
	
	\subsubsection{When $q=1 \mod 4$ or $q = 3 \mod 4$}
	Using \verb*|Sagemath|, we were able to prove that $\psl(2,13)$ acting on the cosets of $D_{12}$ has the EKR property. However, none of the known techniques worked for the groups in line 13. In general, the intersection density of $\psl(2,q)$, for $q = 1 \mbox{ or }3\ \mod 4$ acting on the $2$-subsets of $\pg(1,q)$, is not easy to find. We believe that new techniques are needed to determine the EKR property of these groups. We have compiled in the following table the status of the EKR property for the action of $\psl(2,q)$ acting on the $2$-subsets of $\pg(1,q)$.
	
	\begin{table}[H]
		\begin{tabular}{c|cccccc}
			\hline
			$q$ & $3$ & $5$ & $7$ & $9$& $11$ & $13 $\\
			\hline \hline 
			$q \mod 4$ & $3$ & $1$& $3$&$1$& $3$& $1$\\
			\hline 
			Maximum cocliques & $4$ & $4$& $12$&$8$& $17$& $12$ \\
			\hline
			Order of point-stabilizers & $2$& $4$& $6$& $8$& $10$& $12$\\
			\hline
		\end{tabular}
		\caption{EKR for the group in line 13 for small values.}
	\end{table}
	
	We conjecture the following based on computational results.
	\begin{conj}
		Let $q$ be a prime power and consider the action of $\psl(2,q)$ on the $2$-subsets of the projective line $\pg(1,q)$.
		\begin{enumerate}[(1)]
			\item If $q \equiv 1 \mod 4$, then $\rho(\psl(2,q)) = 1$.
			\item If $q \equiv 3 \mod 4$, then $\psl(2,q)$ does not have the EKR property.
		\end{enumerate}
	\end{conj}
	
	\subsection{Imprimitive case}
	In \cite{2021arXiv210709327H}, it was proved that there are transitive groups of degree a product of two odd primes $pq$, where $p>q$, whose intersection densities are equal to $q$. In this paper, we proved that if an imprimitive group of degree $pq$ has intersection density larger than $1$, then it admits a unique system of imprimitivity with blocks of size $q$. It was believed that the intersection density of transitive groups of degree $pq$ is either $1$ or $q$. Recently, the author, Behajaina and Maleki \cite{imprimitive} constructed a family of imprimitive groups of degree a product of two odd primes $p>q$, where $p=\frac{q^k-1}{q-1}$ for some prime $k<q$, and whose intersection density is $\frac{q}{k}$. An example of such a group is \verb*|TransitiveGroup(39,59)| which has degree $39=3\times 13$ and intersection density equal to $\frac{3}{2}$. This construction relies on the existence of elements of certain order in the permutation automorphism group of cyclic codes of length $p$, and of dimension $k$ over $\mathbb{F}_q$. Therefore, an interesting direction is to study the set $\mathcal{I}_{pq}:=\{ \rho(G) \mid \mbox{ $G \leq \sym(\Omega)$ is transitive of degree $|\Omega|=pq$} \}$. The most natural case to investigate is when $q=3$. By Table~\ref{table:classification}, the primitive groups of degree $3p$, where $p>3$ is an odd prime, contain one of $\alt(7)$ (of degree $21$), $\alt(6)$ (of degree $15$), $\ps(4,2)$ (of degree $15$), $\psl(2,9)$ (of degree $15$), or $\psl(2,19)$ (of degree $91$). It follows from the main result of this paper and from \verb*|Sagemath| that these groups all have intersection density equal to $1$. Hence, all primitive groups of degree $3p$ have the intersection density equal to $1$, for any odd prime $p\geq 5$. The set $\mathcal{I}_{3p}$ then depends on the imprimitive groups of degree $3p$. We end this paper by posing the following question, which is due to Meagher, based on computational results.
	\begin{qst}[Meagher]
		Does the inclusion $\mathcal{I}_{3p} \subset \left\{ 1,\frac{3}{2},3 \right\}$ hold, for any odd prime $p>3$?
	\end{qst}
	
	\vspace{0.5cm}
	\noindent{\textsc{Acknowledgment.}} I would like to thank Karen Meagher for the helpful discussions on the results in this paper (in particular Theorem~\ref{thm:psl-even}). I am also grateful to the referees for helping improve the presentation of the paper and for pointing out certain inconsistencies in an earlier version. This research was done while the author was a Ph.D. student under the supervision of Dr. Karen Meagher and Dr. Shaun Fallat at the Department of Mathematics and Statistics, University of Regina.
	
%	\bibliographystyle{plain}
%	\bibliography{ref-join}

\end{document}